\documentclass[12pt]{amsart}

\usepackage{array}

\usepackage{booktabs}

\usepackage{tabularx}
\usepackage{arydshln}%分块矩阵划线
\usepackage{multirow} %不规则分块矩阵
\usepackage{enumitem}
\usepackage{enumerate}
\usepackage{latexsym}
\usepackage{cite}
\usepackage[centertags]{amsmath}
\usepackage{amsfonts}
\usepackage{dsfont}
\usepackage{amssymb}
\usepackage{amsthm}
\usepackage{newlfont}
\usepackage{graphics}
\usepackage{color}
\usepackage{float}
\usepackage{diagbox}
\usepackage{longtable}
\usepackage{rotating}
\usepackage{multirow}
\usepackage{extarrows}
\usepackage[sort,compress,numbers]{natbib}
\usepackage[utf8]{inputenc}
\usepackage{collectbox}
\usepackage{hyperref}

\numberwithin{equation}{section}
\newtheorem{theo}{Theorem}
\newtheorem{defn}[theo]{Definition}
\newtheorem{exam}[theo]{Example}
\newtheorem{lem} [theo]{Lemma}
\newtheorem{cor}[theo]{Corollary}
\newtheorem{prop}[theo]{Proposition}
\newtheorem{rem}[theo]{Remark}
\newtheorem{algor}[theo]{Algorithm}

\numberwithin{equation}{section}

\makeatletter \@addtoreset{equation}{section}

\def\CT{\mathop{\mathrm{CT}}}
\def\x{\mathbf{x}}
\def\b{\mathbf{b}}
\def\N{\mathbb{N}}
\def\Q{\mathbb{Q}}
\def\Z{\mathbb{Z}}
\def\R{\mathbb{R}}
\def\P{\mathcal{P}}
\def\y{\mathbf{y}}

\def\Vol{\mathrm{Vol}}
\def\vol{\mathrm{vol}}
\def\Det{\mathrm{Det}}

\def\cone{\mathrm{cone}}

\def\Oge{\mathop{\Omega}_\geq}

\def\1{\mathbf{1}}

\title[Algebraic Volume from Ehrhart Theory]{Algebraic Volume for Polytope Arise from Ehrhart Theory}

\author{guoce xin$^{1}$ }
\author{xinyu xu $^{2}$}
\author{yingrui zhang$^{3}$}
\author{zihao zhang$^{4 *}$ }

\address{ $^{1,2,4}$School of Mathematical Sciences, Capital Normal University,
 Beijing 100048, PR China}
 \address{$^{3}$ Academy of Mathematics and Systems Science, Chinese Academy of Sciences ,
 Beijing 100048, PR China}
\email{\newline $^1$\texttt{guoce.xin@163.com }\ \& $^2$\texttt{xinyu0510x@163.com} \&  $^3$\texttt{zyrzuhe@126.com} \& $^4$\texttt{zihao-zhang@foxmail.com}}

\date{December 20, 2023}
\thanks{$*$Corresponding author.}
\thanks{ This work was partially supported by NSFC(12071311).}



\begin{document}

\begin{abstract}
Volume computation for $d$-polytopes $\mathcal{P}$ is fundamental in mathematics. There are known volume computation algorithms, mostly based on triangulation or signed-decomposition of $\mathcal{P}$. We consider $ \mathrm{cone}(\mathcal{P})$ as a lift of $\mathcal{P}$ in view of Ehrhart theory. By using technique from algebraic combinatorics, we obtain a volume algorithm using only signed simplicial cone decompositions of $ \mathrm{cone}(\P)$. Each cone is associated with a simple algebraic volume formula. Summing them gives the volume of the polytope. Our volume formula applies to various kind of cases. In particular,
we use it to explain the traditional triangulation method and Lawrence's signed decomposition method. Moreover, we give a completely new primal-dual method 
for volume computation. This solves the traditional problem in this area:
All existing methods are hopelessly impractical for either the class of simple polytopes or the class of simplicial polytopes. 
Our method has a good performance in computer experiments.
\end{abstract}
\maketitle

%\noindent\subseteq
\begin{small}
 \emph{Mathematic subject classification}: Primary 05A15; Secondary 52B05, 68U05, 52B11.
\end{small}

\noindent
\begin{small}
\emph{Keywords}: Volume; Convex polytope; Ehrhart quasi-polynomials; Constant terms.
\end{small}

\section{introduction}

Polytopes are both theoretically useful and practically essential as we  use them to link results in number theory and combinatorics. We will focus on
rational convex polyhedron defined by
\begin{equation}\label{polytope-A-b}
\P=\{\alpha\in\R^n_{\ge 0} : A\alpha =\b\}, \text{ where } A \in \Z^{r\times n}, \b \in \Z^r.
\end{equation}
The $A$ is usually assumed to be of rank $r$ and $\P$ is of dimension $\dim(\P)=d=n-r$. Rigorously, the dimension of a polytope $\P$ is the dimension of its \emph{affine space}, .i.e. $ \mathrm{span}(\P):=\{\x+\lambda(\mathbf{y}-\x):\x,\mathbf{y}\in\P,\lambda\in\mathbb{R}\}$. One needs some technique condition to have $\dim(\P)=n-r$. When $\P$ is bounded, it is called a rational convex $d$-polytope. We will also use two classical representations of $d$-polytope $\P$ in $\R^d$ in the literature: one is defined as the convex hall of its vertices $V = \{v_1, v_2,\cdots,v_m \}$ in $\R^d$, and the other is defined as the intersection of a finite set of half spaces. The former is called the $\mathcal{V}$-representation given by
$$\P=\textrm{conv}(v_1,\dots, v_m):=\{\lambda_1v_1+\cdots +\lambda_m v_m: \lambda_k\ge 0, \lambda_1+\cdots+\lambda_m=1\};$$
the latter is called the $\mathcal{H}$-representation compactly written as $\P=\{\alpha\in\R^d : A'\alpha \leq \b'\}$ for some matrix $A'$ and vector $\b'$.

Polytope volume is a truly fundamental concept. There are numerous applications of polytope volume computation, ranging from estimating the size of solution space of a linear program to count the number of roots of the system of complex polynomial equations.
It is known \cite{Khachiyan1993} that
computing the volume of rational polytope is strongly \#P-hard.  
Many algorithms for volume computation have been developed but all current algorithms have certain deficiency.
A hybrid method was developed in \cite{bueler2000exact} based on the observation that no method described in the literature works
efficiently on a wide range of polytopes. All methods are hopelessly impractical for either the class of simple polytopes or the class of simplicial polytopes. For example, all triangulation methods work poorly for simple polytopes and Lawrence's signed decomposition method crawls for
simplicial polytopes.

The base stone for volume computation is the formula for a simplex:
\begin{equation}\label{Delta-formula}
 \mathrm{vol} (\Delta(v_0,v_1,\dots,v_d)) =\frac{1}{d!} | \det(v_1-v_0, \dots, v_d-v_0)|,
\end{equation}
where $\Delta(v_0,v_1,\dots,v_d)$ denotes the simplex in $\mathbb{R}^d$ with vertices $v_0,v_1,\dots, v_d$. Almost all known algorithms for exact volume computation rely explicitly or implicitly this formula. For example, the triangulation method is to signed decompose $\P$ into simplices and sum on their volume. See the excellent survey \cite{gritzmann1994complexity} for description of several basic approaches in depth and the relevance of volume computation. A very different approach was given by Lawrence \cite{lawrence1991polytope} without any triangulation when the $d$-polytope is simple (if every vertex is
contained in exactly $d$ facets). By using Gram's relation, Lawrence was able to write the volume $\mathrm{vol}(\P)$ as a sum of the numbers $N_v$, defined for each vertex $v$ of $\P$.

The algebraic volume formula we are going to present unifies the above two formulas. The idea arises from Ehrhart theory, especially the Ehrhart series of a polytope $\P$. See Section \ref{pre-Ehrhart-vol} for details. At this moment, it is convenient to assume $\P$ is integral, i.e., all the vertices of $\P$ are integral. This does not lose generality for volume computation since $\vol(s\P)=s^d \vol(\P)$, where $s\P=\{s\alpha: \alpha\in \P\}$  is the $s$ \emph{dilation} of $\P$.

We use a process called \emph{coning over a polytope}: the cone over $\P$ is defined by
$$ \mathrm{cone}(\P)=\{(\alpha,s)\in \R_{\geq 0}^{n+1}: \alpha\in s\P\}\longrightarrow \{(\alpha,s)\in \R_{\geq 0}^{n+1}: A\alpha=s\b\}.$$
By cutting $ \mathrm{cone}(\P)$ with the hyperplane $x_{n+1}=s$, we obtain $s\P=\{s\alpha: \alpha\in \P\}
\longrightarrow \{\alpha\in\mathbb{R}^n_{\geq0}:A\alpha=s\b\}$.
Ehrhart \cite{ehrhart1974polyn} first studied the function
$$L_{\P}(s):= \# (s \P \cap \mathbb{Z}^n) $$
and confirmed that $L_{\P}(s)$ is a polynomial of degree $d$,  when $\P$ is integral. Moreover, This polynomial is now called the Ehrhart polynomial
and its leading coefficient is known to be equal to $\vol(\P)$.
In terms of generating functions, this is equivalent to saying that the Ehrhart series of $\P$ is of the form
$$Ehr_{\P}(t)= 1+\sum_{s\geq1}L_{\P}(s)t^s =\frac{N(t)}{(1-t)^{d+1}},$$
where $N(t)$ is a polynomial of degree no more than $d$, and
$\vol(\P) = \frac{1}{d!} N(1)$.  This formula has been used for volume computation. See Section \ref{sec:ConcludingRemark}.

Our starting point is the formula
$$\mathrm{vol}(\P)=\frac{1}{d!}\Vol_d Ehr_{\P}(t),$$
where $\Vol_d$ is the linear operator acting on rational functions $Q(t)$ with pole at $t=1$ of multiplicity at most $d+1$ by
$$ \mathrm{Vol}_d Q(t): = Q(t)\cdot (1-t)^{d+1}  \Big|_{t=1}.$$
In order to take advantage of the above formula, we need a suitable representation of $Ehr_{\P}(t)$.
Consider the multivariate Ehrhart series $Ehr_{\P}(\y;t)$ defined by
\begin{equation}\label{Ehr-cone}
 Ehr_{\P}(\y;t)   =\sum_{s=0 }^\infty \sum_{\alpha \in s\P  \cap \N^n} \y^\alpha  t^s,
\end{equation}
where $\y^\alpha$ is short for $y_1^{\alpha_1} y_2^{\alpha_2} \cdots y_n^{\alpha_n}$. In fact, $ Ehr_{\P}(\y;t)$ corresponds to the generating function of $ \mathrm{cone}(\P)$.
By using the constant term method in \cite{xin2015euclid}, we can obtain a decomposition of the form
\begin{align}\label{equ-Ehrseries}
 Ehr_{\P}(\y;t)=\sum\limits_i F_i(\y;t)=\sum_i  \frac{L_i(\y;t)}{\prod_{k=1}^{h_i}(1-\y^{\nu_{i,k}} t^{m_{i,k}})},
\end{align}
where $h_i\leq d+1$, $\nu_{i,k}$ are vectors in $\mathbb{Z}^{n}$, $m_{i,k}\in \mathbb{Z}$, and $L_i(\y;t)$ are Laurent polynomials.
Furthermore, $Ehr_{\P}(t)=Ehr_{\P}(\1;t)$ is computed through a ``dispelling the slack variables" process.
Careful analysis of this process leads to a formula of $\vol(\P)$.
\begin{defn}\label{defn-general-F}
Let $F(\y;t)=\frac{L(\y;t)}{\prod_{k=1}^{h}(1-\y^{\nu_{k}} t^{m_{k}})}$ with $h\leq d+1$. A vector $\beta\in \Z^n$ is said to be \emph{admissible} for $F(\y;t) $ if $\beta^T \cdot \nu_{k} $ and $m_{k}$ are not both $0$s. The $\beta$ algebraic volume of $F(\y;t)$ is defined by
\begin{equation}\label{equ-Vol_d}
\Vol_d^{\beta} F_i(\y;t):=\left\{
\begin{aligned}
&\CT_q \frac{ L_i(\mathbf{1};1)}{ \prod_{k=1}^{d+1}(m_{i,k}-(\beta^T \cdot \nu_{i,k})  q)}&,\quad h_i=d+1\\
 &0&,\quad h_i<d+1,
\end{aligned}
\right.
\end{equation}
where $\CT\limits_q$ means to take constant term of a Laurent series in $q$.
\end{defn}

Our first result can be stated as follows.
\begin{theo}\label{theo-vol-formula}
Suppose $Ehr_{\P}(\y;t)$ is given in \eqref{equ-Ehrseries}. If $\beta\in \Z^n$ is \emph{admissible} for all $F_i$, then we can get
$$\mathrm{vol}(\P)=\frac{1}{d!}\sum\limits_i \Vol_d^{\beta} F_i(\y;t),$$
where $\Vol_d^{\beta} F_i(\y;t)$ is defined in Equation \eqref{equ-Vol_d}.
\end{theo}

When $\P$ is of full dimensional, we only need a vertex simplicial cone decomposition to compute $\vol(\P)$,  where a \emph{vertex cone} is of the form $$K^v=K^v(\mu_1,\mu_2,\dots,\mu_h):= \{ v+k_1 \mu_1+\cdots +k_h \mu_h: k_i\ge 0\}$$
 with $h \leq d+1$ and $\mu_j=(\nu_j^T,m_j)^T\in \Q^{d+1}$. The superscript $v$ will be omitted when it is the origin. Thus $K^v=v+K(\mu_1,\mu_2,\dots,\mu_h)$.
\begin{defn}
Let $K^v(\mu_1,\mu_2,\dots,\mu_h)$ be a vertex simplicial cone, where $h \leq d+1$, and $\mu_j=(\nu_j^T,m_j)^T\in \R^{d+1}$. Define
\begin{equation}
  \label{e-vol-ful-simplical-cone}
\Vol_d^\beta K^v(\mu_1,\mu_2,\dots,\mu_h) :=\left\{
                      \begin{array}{ll}
                       |\det(\mu_1,\dots, \mu_{d+1} ) |\CT\limits_q \frac{1}{ \prod_{i=1}^{d+1}(m_{i}-(\beta^T \cdot \nu_{i})  q)}&,  \mbox{if } h=d+1;\\
                           0 &, \mbox{if } h<d+1,
                      \end{array}
                    \right.
\end{equation}
where $\beta$ is admissible for $K^v(\mu_1,\mu_2,\dots,\mu_h)$ when $\beta^T \cdot \nu_j$ and $m_j$ are not both $0$s for all $j$.
\end{defn}

Note that the formula is invariant under the replacement of $\mu_j$ by $p \mu_j$ for any $j$ and $p>0$.
We have
\begin{theo}\label{general-full-volcompute}
Let $\P$ be a (not necessarily rational) $d$-polytope in $\R^d$.
Suppose $\{(s_i,K_i^{v_i})\}$ is a signed simplicial cone decomposition of $ \mathrm{cone}(\P)$, where $\dim(K_i^{v_i})\leq d+1$.
Then the volume $\mathrm{vol}(\P)$ is equal to
$$\mathrm{vol}(\P)=\frac{1}{d!}\sum_{i=1}^{N} s_i\Vol_d^\beta K_i^{v_i} ,$$
for any $\beta$ admissible for all $K_i^{v_i}$.
\end{theo}
Note that our decomposition of $\mathrm{cone}(\P)$ induces a polyhedron decompositions of $\P$ by intersecting each $K^v(\mu_1,\mu_2,\dots,\mu_h)$ with the hyperplane $s=1$.
The intersection could be a Minkowski sum of the form $\textrm{conv}(\nu_1,\dots,\nu_\ell)+K(\nu_{\ell+1},\cdots,\nu_{d+1})$.
This is easily seen to be the case when $v$ is the origin, $m_i=1$ for $1\leq i\leq \ell$ and $m_i=0$ for $i>\ell$. The Minkowski sum reduces to a simplex
when $\ell=d+1$, and reduces to a vertex cone when $\ell=1$.
Therefore our volume formula includes the simplex volume formula and Lawrence' formula when $\P$ is simple as special cases. See Section \ref{sec:sign-volformula}.
Figure \ref{2-cube-dec} illustrates an example of polyhedron decomposition. We will also illustrate how to compute the volume in the dual space.
\small{
\begin{figure}[!htbp]
$$
  \hskip .1 in \vcenter{ \includegraphics[height=1.4 in]{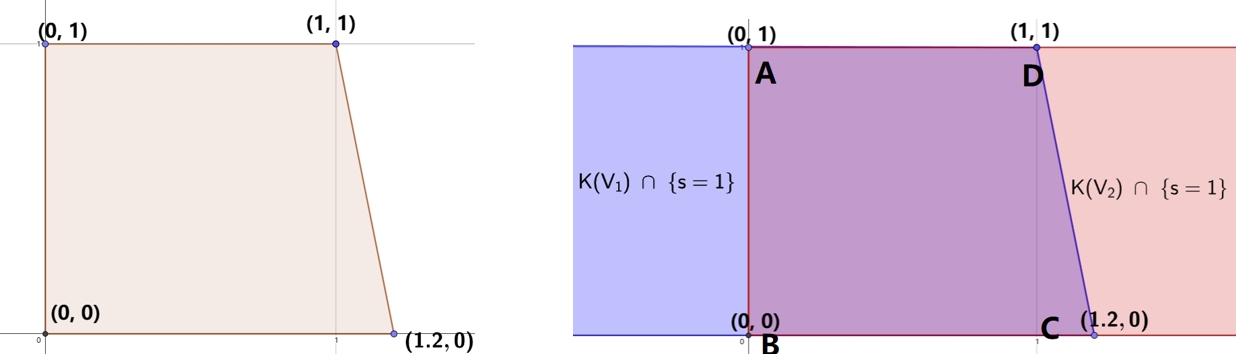}}
$$
\caption{A perturbed square and its polyhedron decomposition. The two polyhedrons on the right
are $\textrm{conv}(A,B)+K((1,0))$ and $\textrm{conv}(C,D)+K((-1,0))$.
}\label{2-cube-dec}
\end{figure}
}

We also have an extension of Theorem \ref{general-full-volcompute} for $d$-polytopes in $\R^n$. Especially, when $\P$ is given by \eqref{polytope-A-b},
we give a direct formula in Theorem \ref{thm-rvol} for $\vol(\P)$ without computing the determinants. Such a formula is valuable under the observation that Stanley actually developed a simplicial cone decomposition when developing his monster reciprocity theorem. See Section \ref{sec:simpcone}.

The paper is organized as follows. In Section \ref{pre-Ehrhart-vol}, we present some necessary preliminaries, including an introduction to the Ehrhart theory and its relation to volume. In Section \ref{volume formula}, we prove Theorem \ref{theo-vol-formula}, which is a basic result for volume computation. It arises from the Ehrhart theory. The formula applies for any decomposition as in \eqref{equ-Ehrseries}.
In Section \ref{sec:sign-volformula}, we observe that it is sufficient to use signed simplicial cone decomposition of $ \mathrm{cone}(\P)$. We first prove Theorem \ref{general-full-volcompute} for the full dimensional case, which applies to general polytopes. Next we prove Theorem \ref{general-volcompute} for rational $d$-polytopes in $\R^n$. Finally we use Theorem \ref{general-full-volcompute} to re-prove two well known formulas,
 one for signed simplices decompositions and the other for Lawrence's volume formula. Section \ref{sec:simpcone} handles polytope defined by $A\alpha=\b$. We first introduce a new combinatorial simplicial cone decomposition of $ \mathrm{cone}(\P)$.
Then we give a direct formula in Theorem \ref{thm-rvol} for $\vol(\P)$ without computing the determinants.
Section \ref{sec:primal-dual} describes how to compute the volume of a polytope in the dual space, and hence giving a desired solution for both simple polytope and simplicial polytope. Section \ref{sec:compexp} discuss some compute experiments.
Section \ref{sec:ConcludingRemark} is a concluding remark, where we talk about several existing volume computation methods using Ehrhart theory.

\section{Preliminary}\label{pre-Ehrhart-vol}
We briefly introduce the Ehrhart theory and its relation to volume. We follow notation in \cite{beck2007computing}.

For a $d$-polytope $\P\subset\R^n$, we want to compute the volume relative to the sublattice $ \mathrm{span}(\P)\cap\Z^n$. This is called the \emph{relative volume} of $\P$ and it is known to be equal to
$$\mathrm{vol}(\P)=\lim\limits_{s\rightarrow\infty}\frac{1}{s^d}\cdot\#(s\P\cap\Z^n).$$

For example, the line segment $L$ from $(0,0)$ to $(4, 2)$ in $\R^2$  has relative
volume $2$, because in  $ \mathrm{span} L=\{(x,y)\in R^2:  y=x/2\}$, $L$ is covered by two segments of ``unit length" $\sqrt{5}$ in this affine subspace, as pictured in Figure \ref{fig}.

If $\P$ in $\R^d$ is full dimensional, then the relative volume coincides with the volume defined by  the integral $\mathrm{vol}(\P):=\int_\P dx$. See detail in \cite[Chap 5]{beck2007computing}.

\begin{figure}[htb]
  \centering
  \includegraphics[width=0.4\linewidth]{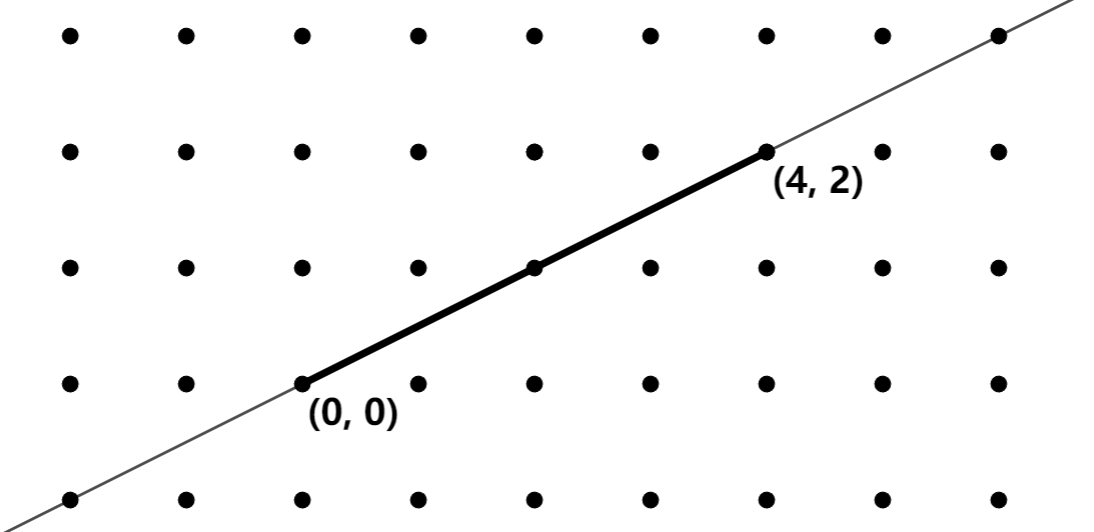}
\caption{Line segment from $(0,0)$ to $(4,2)$ and its affine space.
\label{fig}}
\end{figure}

The function
$$L_{\P}(s):= \# (s \P \cap \mathbb{Z}^n) \longrightarrow \# \{\alpha \in \mathbb{Z}^n: A\alpha =s \b,\; \alpha \ge 0 \}$$
defined for nonnegative integer $s$ was first studied by E. Ehrhart \cite{ehrhart1974polyn}, and is called the Ehrhart quasi-polynomial of $\P$. It is a polynomial when $\P$ is integral. It has been shown that
$$L_{\P}(s)= c_d s^d + c_{d-1} s^{d-1}+ \cdots +c_0, $$
where $c_i$ are periodic functions in $s$. Some coefficients are known. For instance $c_d$ is a constant and is the relative volume of $\P$
since
$$\mathrm{vol}(\P) =\lim\limits_{s\rightarrow\infty}\frac{1}{s^d}\cdot\#(s\P\cap\Z^n)= \lim_{s\to \infty}\frac{1}{s^d} {L_{\P}(s) }=c_d.$$

In terms of generating functions,
the Ehrhart series of $\P$ is defined by
$$Ehr_{\P}(t)= 1+\sum_{s\geq1}L_{\P}(s)t^s =\sum_{s\geq 0}L_{\P}(s)t^s .$$
In Ehrhart theory, $Ehr_{\P}(t)$ is known to be a proper  rational function of the form:
$$ Ehr_{\P}(t) = \frac{N(t)}{(1-t) (1-t^p)^{d}},$$
where $p$  is the least positive integer such that $p\P$ is integral, and $N(t)$ is a polynomial of degree less than $pd+1$.
Now comes the starting point of this work. We have
$$\mathrm{vol}(\P) =\frac{1}{d!}  Ehr_{\P}(t)\cdot (1-t)^{d+1} \Big|_{t=1}.$$

Define the linear operator $\Vol_d$ acting on rational functions $Q(t)$ having $t=1$ as a pole of multiplicity at most $d+1$ by
$$ \mathrm{Vol}_d Q(t): = Q(t)\cdot (1-t)^{d+1}  \Big|_{t=1}.$$
Then $\mathrm{vol}(\P)=\frac{1}{d!}\Vol_d Ehr_{\P}(t)$. The advantage of this new formula is that we can compute the volume separately
by means of the decomposition \eqref{equ-Ehrseries} of the multivariate Ehrhart series $Ehr_{\P}(\y;t)$ defined as in \eqref{Ehr-cone}.

\section{Algebraic Volume and Polytope Volume}\label{volume formula}
We justify Definition \ref{defn-general-F} for each term in \eqref{equ-Ehrseries} and prove Theorem \ref{theo-vol-formula}.

We  use the following multivariate generating functions of any set $S \subset \R^n$:
$$\sigma_{S}( \y)=\sum_{\alpha \in S \cap \Z^n } \y^\alpha .$$
Then the multivariate Ehrhart series is just $Ehr_{\P}( \y;t)=\sigma_{ \mathrm{cone}(\P)}(\y)$, where we set $y_{n+1}=t$.

Computing the Ehrhart series or Ehrhart quasi-polynomial is an important problem and is harder than computing the volume. Earlier method computes $L_{\P}(s)$ for sufficiently many $s$ and then uses the Lagrange interpolation formula to construct $L_{\P}(s)$.

Here we closely follow the CTEuclid algorithm developed in \cite{xin2015euclid}, where it is shown that
$Ehr_{\P}(t)$ can be computed directly through constant term extraction.

In order to illustrate the idea, we use the operator $\CT$.
\begin{defn}
 % $$\Oge \sum_{s_1=-\infty}^{\infty} \cdots \sum_{s_r=-\infty}^{\infty} a_{s_1,\dots,s_r} \lambda_1^{s_1}\dots \lambda_r^{s_r}=\sum_{s_1=0}^{\infty} \cdots \sum_{s_r=0}^{\infty} a_{s_1,\dots,s_r},$$
   $$\CT_\Lambda \sum_{s_1=-\infty}^{\infty} \cdots \sum_{s_r=-\infty}^{\infty} a_{s_1,\dots,s_r} \lambda_1^{s_1}\dots \lambda_r^{s_r}= a_{0,\dots,0}.$$
  Every rational function has a unique iterated Laurent series representation \cite{xin2004fast}.
\end{defn}

For $\P$ specified by $A$ and $\b$ as in \eqref{polytope-A-b}, we have the  constant term expression:
\begin{align}\label{eq-defEhr}
Ehr_{\P}(\y;t) = \CT_{\Lambda} \frac{1}{
\prod_{i=1}^n (1- \Lambda^{A_i} y_i)(1-\Lambda^{-\b} t )},\  \text{where $A_i$ is $i$th column of $A$}.
\end{align}
Firstly \texttt{CTEuclid} write $Ehr_{\P}(\y;t)$ as a sum of simple rational function as in  \eqref{equ-Ehrseries}. Next we describe how to compute $Ehr_{\P}(t)=Ehr_{\P}(\1;t)$ by the ``dispelling slack variables" process. A careful analysis can give rise a formula of $\vol(\P)$.

Calculating $Ehr_{\P}(\1;t)$ is equivalent to computing the limit of $Ehr_{\P}(\y;t)$ at $y_k=1$ for all $k$.
Direct substitutions by $y_k= 1$ for all $k$ do not work for possible denominator factors like $ 1-y_1 y_2$ in some of the terms. Dispelling the slack variables $y_k$ for all $k$ consists of two major steps.

Step 1 is to reduce the number of slack variables to 1.
This is done by finding a suitable integer vector $\beta=(\beta_1,\dots,\beta_n)^T$ and making the substitution $y_i\to \kappa^{\beta_i}$ to obtain
\begin{equation}\label{e-tok}
  Ehr_{\P}(\kappa^{\beta_1},\dots,\kappa^{\beta_n};t)=\sum_{i} F_i(\kappa^{\beta_1},\dots,\kappa^{\beta_n};t).
\end{equation}
However, for this substitution to work, $\beta$ must be picked such that there is no
zero in the denominator of each $F_i$. We call such $\beta$ \emph{admissible} (for all $F_i$). Barvinok showed that admissible  $\beta$ can be picked in
polynomial time by choosing points on the moment curve. De Loera et al. \cite{de2004effectiveLattE} suggested using
random vectors to avoid large integer entries.

Step 2 is to use Laurent series expansion and take constant term. By making the exponential substitution $\kappa=e^s$, we arrive at
$$Ehr_{\P}(\kappa^{\beta_1},\dots,\kappa^{\beta_n};t)\Big|_{\kappa=1}=Ehr_{\P}(e^{s\beta_1},\dots,e^{s\beta_n};t)\Big|_{s=0}=\CT_s Ehr_{\P}(e^{s\beta_1},\dots,e^{s\beta_n};t)
.$$

 The linearity of the operator $\CT_s$ allows us to compute separately:
$$Ehr_{\P}(t)=  \CT_s Ehr_{\P}(e^{s\beta_1},\dots,e^{s\beta_n};t)=\sum_{i} \CT_s F_{i}(e^{s\beta_1},\dots,e^{s\beta_n};t).$$
It follows that
\begin{equation} \label{e-Vd-beta-EHr}
\Vol_d Ehr_{\P}(t)=\sum\limits_i \mathrm{Vol}_d \CT\limits_s F_i(e^{s\beta_1},\dots,e^{s\beta_n};t).
\end{equation}

%
% \textbf{Note}: Each term $F_i^\beta(t)$ is closely related to the Todd class of a toric variety and it can be fast evaluated. See \cite{} and the references therein. Here we derive a nice formula of $\mathrm{Vol}_d F_i^\beta(t)$.

\begin{prop}\label{p-denominator-volume}
Consider the constant term
$$C(t)=\CT\limits_s \frac{L(e^s;t)}{\prod_{j=1}^{d+1} (1-e^{sb_j}t^{u_j})},$$ where $L(\kappa;t)$ is a Laurent polynomial.
If $u_1=u_2=\cdots =u_r=0$ and $u_j\neq0$ for all $j>r$, then we have
$$\mathrm{Vol}_d C(t) =  [q^r] \frac{L(1;1)}{  \prod_{i=1}^r(-b_i) \prod_{j=r+1}^{d+1} (u_j-b_j q) }=\CT\limits_q \frac{L(1;1)}{\prod_{j=1}^{d+1} (u_j-b_j q) }.$$
\end{prop}
\begin{proof}[Proof of Proposition \ref{p-denominator-volume}]
We use the explicit summation formula for $C(t)$ given in \cite{xin2015euclid}
\begin{align}\label{e-single-ct}
 \sum_{n_0+n_1+n_{r+1}+\cdots +n_{d+1} =r} \ell_{n_0}(t) c_{n_1}' \prod_{j=r+1}^{d+1} c_{n_j}(M_j)b_j^{n_j},
\end{align}
where $M_j=t^{u_j}$ and $L( e^{s};t) = \sum_{n\ge 0} \ell_n(t) s^n,$
$$ \prod_{j=1}^r \frac{ s}{1-e^{b_js}}=\prod_{j=1}^r \sum_{n\ge 0}- \frac{\mathcal B_nb_j^{n-1}}{n!} s^n= \sum_{n\ge 0} c'_n s^n ,$$
with $\mathcal{B}_n$ being the well-known Bernoulli numbers, and
$$ \frac{1}{1-e^{s}M}  = \sum_{n\ge 0} c_n(M) s^n.$$
Particularly, we have $\ell_{0}(1)=L(e^{0};1)=L(1;1)$ and $c'_0= \prod_{i=1}^r(-b_i)^{-1}$.

The explicit formula for $c_n(M)$ can be obtained by using Stirling number of the second kind $S(n,k)$, %defined by
\begin{align*}
\frac{1}{1-e^{s}M}&= \frac{1}{(1-M) (1- \frac{M}{1-M}(e^{s}-1))}
= \frac{1}{(1-M)} \sum_{k\ge 0} \frac{M^k}{(1-M)^{k}} (e^{s}-1)^k \\
&= \sum_{k\ge 0} \frac{M^k}{(1-M)^{k+1}}k! \sum_{n\ge k} S(n,k) \frac{s^n}{n!}
= \sum_{n\ge 0} \left( \sum_{k=0}^n  \frac{k!}{n!} S(n,k) \frac{M^k}{(1-M)^{k+1}} \right) s^n .
\end{align*}
So
$$ c_n(M) = \sum_{k=0}^n  \frac{k!}{n!} S(n,k) \frac{M^k}{(1-M)^{k+1}} $$
has denominator $(1-M)^{n+1}$ and then  $c_n(t^u)$ has $t=1$ as a pole of multiplicity $n+1$.

Therefore any term in \eqref{e-single-ct} with respect to $n_0+n_1+n_{r+1}+\cdots +n_{d+1} =r$, i.e.,
$$ \ell_{n_0}(t) c_{n_1}' \prod_{j=r+1}^{d+1} c_{n_j}(M_j)b_j^{n_j},$$
has $t=1$ as a pole of multiplicity at most
$ (n_{r+1}+1)+\cdots +(n_{d+1}+1) = d+1-n_0-n_1.$
When applying the operator $\mathrm{Vol}_d$, such a term survives only when
$n_0=0, n_1=0$. In this case, we have

\begin{align*}
\mathrm{Vol}_d&\left(\ell_{0}(t) c_0' \prod_{j=r+1}^{d+1} \sum_{k=0}^{n_j}  \frac{k!}{n_j!} S(n_j,k) \frac{t^{u_jk}}{(1-t^{u_j})^{k+1}} b_j^{n_j}\right) \\
&= c_0'\ell_{0}(t)\left(\prod_{j=r+1}^{d+1} \sum_{k=0}^{n_j}  \frac{k!}{n_j!} S(n_j,k) \frac{t^{u_jk}}{(1-t^{u_j})^{k+1}} b_j^{n_j}\right)\cdot (1-t)^{d+1}\bigg|_{t=1}\\
&=c_0'\ell_{0}(t)\prod_{j=r+1}^{d+1}\left(\sum_{k=0}^{n_j}  \frac{k!}{n_j!} S(n_j,k) \frac{t^{u_jk}}{(1-t^{u_j})^{k+1}}b_j^{n_j}\cdot (1-t)^{n_j+1}\right)\bigg|_{t=1} \\
&=c_0' \ell_{0}(1)\left(\prod_{j=r+1}^{d+1} \frac{b_j^{n_j}}{u_j^{n_{j}+1}}\right).
\end{align*}
Taking the sum over all  $ (n_{r+1}+1)+\cdots +(n_{d+1}+1) = d+1$ gives
\begin{align*}
\mathrm{Vol}_d& \sum_{n_{r+1}+\cdots +n_{d+1} =r}  \ell_{0}(t) c_0' \prod_{j=r+1}^{d+1} \left(\sum_{k=0}^{n_j}  \frac{k!}{n_j!} S(n_j,k) \frac{t^{u_jk}}{(1-t^{u_j})^{k+1}}\right) b_j^{n_j} \\
&=c_0' \ell_{0}(1) \sum_{n_{r+1}+\cdots +n_{d+1} =r}\left(\prod_{j=r+1}^{d+1} \frac{b_j^{n_j}}{u_j^{n_{j}+1}}\right)\\
&=c_0'\ell_{0}(1) [q^r] \prod_{j=r+1}^{d+1} \frac{1}{u_j} \frac{1}{1-\frac{b_j}{u_j} q} \\
&= c_0'\ell_{0}(1) [q^r] \prod_{j=r+1}^{d+1} \frac{1}{u_j-b_j q}= \CT\limits_q \frac{\ell_{0}(1)}{\prod_{j=1}^{d+1} (u_j-b_j q)}.
\end{align*}
The proposition then follows since $ c_0'=\prod_{i=1}^r(-b_i)^{-1}$ and $\ell_0(1)= L(1;1)$.
\end{proof}

Now we are ready to prove Theorem \ref{theo-vol-formula}.
\begin{proof}[Proof of Theorem \ref{theo-vol-formula}]
From the decomposition of $Ehr_{\P}(\y;t)$ in \eqref{equ-Ehrseries}, we obtain \eqref{e-Vd-beta-EHr}. Applying Proposition \ref{p-denominator-volume}
to the $i$-th term gives
$$ \mathrm{Vol}_d \CT\limits_s F_i(e^{s\beta_1},\dots,e^{s\beta_n};t)=\Vol_d^\beta F_i(\y;t).$$
Summing over all $i$ completes the proof.
\end{proof}

\begin{rem}
 The coefficient  $[q^r] \frac{1}{ \prod_{j=r+1}^{d+1} (1-b_j q/u_j) }$ can be computed using $r(d-r+1)+O(r\log r)$ multiplications and some additions, using a logarithmic and exponential technique in \cite{XinTodd}.
\end{rem}

\section{A volume formula by a given signed simplicial cone decomposition}\label{sec:sign-volformula}
The $F_i(\y;t)$ in Theorem \ref{theo-vol-formula} can be taken as the generating function of a certain simplicial cone, since we only need
$L_i(\1;1)$. This allows us to give a formula directly from a simplicial cone decomposition of $ \mathrm{cone}(\P)$. Moreover, $\P$ need not be full dimensional.

\subsection{Notation and the Full Dimensional Case}
We first introduce some notation. Let $w_1,\dots,w_m \in \Q^n$ be the column vectors of a matrix $W$. The cone generated by these vectors is denoted
$$K=K(W)=K(w_1,...,w_m)= \{k_1w_1+k_2w_2+\dots+k_m w_m:\ k_1,k_2,\dots,k_m \in \R_{\ge0} \}.$$
We require this set do not contain a straight line, and usually assume the generators are minimal (also called extreme rays).
Denote by $K^v=v+K=\{v+\alpha: \alpha\in K\}$ the translation of $K$ by a vector $v\in \Q^n$. We call
$K^v(w_1,...,w_m)$ a \emph{vertex cone} with vertex $v$ and generators $w_1,\dots,w_m$.
If $w_1,...,w_m$ are linearly independent, we call $K=K^{\mathbf{0}}$ a simplicial cone and $K^v$ a vertex simplicial cone.

A simplicial cone has a unique representation by its primitive generators. Any rational nonzero vector $\alpha$ is associated with
a unique primitive vector $\bar{\alpha}=p \alpha$ for some $p >0$, where by primitive, we mean $\bar{\alpha}$ is integral and its entries have greatest common divisor $1$. Let $K^v=v+K(\bar\mu_1,\dots,\bar\mu_m)$.
A well known result of Stanley \cite{stanley2011EC1} states that
$$\sigma_{K^v}(\y)= \frac{\sigma_{v+\Pi(K)}(\y)}{\prod_{k=1}^{m} (1-\y^{\bar\mu_{k}})},$$
where  $\Pi(K)=\{ k_1 \bar\mu_{1} +\cdots+ k_{m} \bar\mu_m : 0 \le k_1,\dots,k_m < 1 \}$ is called the \emph{fundamental parallelepiped} of $K(\bar\mu_1,\dots,\bar\mu_m)$.
We shall use the fact $\sigma_{\Pi(K)}(\1)=\sigma_{v+\Pi(K)}(\1)$.
Moreover, if $K$ is full dimensional, i.e., $m=n$, then the number of lattice points in the parallelepiped is equal to $\sigma_{\Pi(K)}(\1)=|\det((\bar\mu_1,\bar\mu_2, \dots, \bar\mu_{m}))|$.

A collection $\{(s_i,K_i^{v_i})\}$, where $s_i \in \Z$ and $K_i^{v_i}$ is a vertex simplicial cone, is called a signed vertex simplicial cone decomposition of $K^v$ if $\sigma_{K^v}(\y)=\sum_{i} s_i \sigma_{K_i^{v_i}}(\y)$. Our volume formula relies on a cone decomposition of $ \mathrm{cone}(\P)$ in $\R^{n+1}$ as follows:
\begin{equation}\label{decom-polytope}
Ehr_{\P}(\y;t)=\sigma_{ \mathrm{cone}(\P)}(\y;t) =\sum_is_i \sigma_{K_i^{v_i}}(\y,t),
\end{equation}
 where we use $t$ for $y_{n+1}$.
Now we focus on the full dimensional ($n=d$) case.

\begin{proof}[Proof of Theorem \ref{general-full-volcompute}]
Since the formula is continuous in $\mu_i$ for all $i$, it suffices to assume $\P$ is a rational polytope.

By Theorem \ref{theo-vol-formula} and \eqref{decom-polytope}, we only need to prove
 $\Vol_d^{\beta} \sigma_{K_i^{v_i}}(\y,t)= \Vol_d^\beta K_i^{v_i}$ for $h=d+1$.
For simplicity, we shall omit the subscript $i$, so let
 $K_i^{v_i}:=K^v= K^v(\mu_1, \mu_2,\dots,\mu_{d+1})$, where $\mu_j=(\nu_j^T,m_j)^T\in \Q^{d+1}$.
Write $(\bar\mu_1, \bar\mu_2,\dots,\bar\mu_{d+1})=(p_1\mu_1, p_2\mu_2,\dots,p_h\mu_{d+1})$ for some $p_j>0$.

Now
$$\sigma_{K^v}(\y;t)= \frac{\sigma_{v+\Pi(K)}(\y;t)}{\prod_{k=1}^{d+1} (1-(\y,t)^{\bar\mu_{k}})}. $$
Applying Theorem \ref{theo-vol-formula} gives
\begin{align*}
\Vol_d^{\beta} \sigma_{K^v}(\y;t)&=\CT_q \frac{\sigma_{v+\Pi(K)}(\1;1)}{ \prod_{k=1}^{d+1}(\bar m_{k}-(\beta^T \cdot \bar\nu_{k})  q)}=\CT_q \frac{|\det((\bar\mu_1,\bar\mu_2 \dots, \bar\mu_{d+1}))|}{ \prod_{k=1}^{d+1}(\bar m_{k}-(\beta^T \cdot \bar\nu_{k})  q)}\\
&=\CT_q \frac{p_1p_2\cdots p_{d+1}|\det((\mu_1,\mu_2 \dots, \mu_{d+1}))|}{ \prod_{k=1}^{d+1}(p_km_{k}-(\beta^T \cdot p_k\nu_{k})  q)}\\
&=|\det((\mu_1,\mu_2 \dots, \mu_{d+1}))|\CT_q \frac{1}{ \prod_{k=1}^{d+1}( m_{k}-(\beta^T \cdot \nu_{k})  q)}=\Vol_d^\beta K^v.
\end{align*}
The theorem then follows.
\end{proof}

\subsection{Two corollaries}
Here we use Theorem \ref{general-full-volcompute} to re-prove two well known formulas,
one is given by Corollary \ref{Cor-simplice} for signed simplices decompositions and the other is
Corollary \ref{Lawrence1991-volume} for Lawrence's volume formula.

\begin{cor}\label{Cor-simplice}
Let $\P\subseteq\R^d$ be a $d$-polytope. If $\P$ has a signed simplices decomposition  $\{( s_i,\Delta_i): i = 1,\dots, m\}$,
 where $\Delta_i$ are simplices and $s_i= {\pm 1}$ are signs.
Then $$\vol(\P)=\sum\limits_{i=1}^m s_i\mathrm{vol}(\Delta_i),$$
where $\vol(\Delta_i)$ is computed by \eqref{Delta-formula}.
\end{cor}
\begin{proof}
For any signed decomposition $\{( s_i,\Delta_i): i = 1,\dots, m\}$, the coning process naturally give rise
a signed simplicial cone decompositions $\{( s_i, K(V_i)): i = 1,\dots, m\}$, where
if $\Delta_i$ has vertices $v_1,\dots, v_{d+1}$, then $V_i=\left(
             \begin{array}{cccc}
               v_1 & v_2& \cdots & v_{d+1} \\
               1 & 1&\cdots &  1 \\
             \end{array}
           \right).
  $
By taking $\beta=\mathbf{0}$ as the admissible vector, we have
\begin{align*}
 \mathrm{vol}(\P)&=\frac{1}{d!}\sum_{i=1}^{m} s_i \Vol_d^\beta V_i= \frac{1}{d!}\sum\limits_{i=1}^m s_i |\det(V_i)| \CT\limits_q
\frac{1}{\prod_{j=1}^{d+1} (1-(\beta^T\cdot v_j) q)}\\
&=\frac{1}{d!}\sum\limits_{i=1}^m s_i | \det(v_2-v_1, \dots, v_{d+1}-v_1)|
=\sum\limits_{i=1}^m s_i\mathrm{vol}(\Delta_i).
\end{align*}
Then we get the desired result.
\end{proof}

For the volume formula of Lawrence, we need some notation.
 Let $\P$ be full dimensional and integral. For any vertex $v$ of $\P$, the \emph{supporting cone} $K(\P, v)$ of $P$ at $v$ is
$K(\P, v) = v +\{ u \in  \R^d : v +  \delta u \in \P \text{~for all sufficiently small~}  \delta <1 \}$.
A seminal
work of Brion \cite{brion1988points} and independently Lawrence \cite{lawrence1991polytope} is the following:
\begin{theo}\cite{brion1988points} \label{thm-brion}
  Let $\P$ be an integral polyhedron and let $V(\P)$
be the vertex set of $\P$. Then
   $$\sigma_\P(\y) =\sum_{ v \in V(\P)}\sigma_{K(\P, v)}(\y).$$
\end{theo}

\begin{lem}\label{lem-simple}
 Suppose a polytope $\P$ is  simple and integral. Then we have
$$Ehr_{\P}(\y;t)=\sum_{v}  \sigma_{K^0(\P,v)} (\mathbf{y}) \times (1-\y^v t)^{-1},$$
where $K^0(\P,v)=K(\P,v)-v $ denotes the translation of $K(\P,v)$ with vertex at the origin.
\end{lem}
\begin{proof}
Recall that if $v$ is a vertex of $\P$, then $sv$ is a vertex of the dilated polytope $s\P$.
It is simple to verify that $K^0(\P,v)$ = $K^0(s\P,sv)$. Thus
$$\sigma_{K(s\P,sv)}(\mathbf{y})= y^{sv} \sigma_{K^0(s\P,sv)} (\mathbf{y})= y^{sv} \sigma_{K^0(\P,v)} (\mathbf{y}).$$

 By Theorem \ref{thm-brion}, we have
$$\sigma_{s\P}(\mathbf{y})=\sum\limits_{sv\in V(s\P)}\sigma_{K(s\P,sv)}(\mathbf{y})=\sum\limits_{v\in V(\P)} y^{sv} \sigma_{K^0(\P,v)} (\y).$$
Therefore, we obtain
\begin{eqnarray*}
Ehr_P(\y;t)&=&\sum_{s\geq0}\sigma_{s\P}(\y)t^s=\sum_{s\geq0}\sum_{v} y^{sv} \sigma_{K^0(\P,v)} (\y) t^s\\
 &=&\sum_{v}\sigma_{K^0(P,v)} (\y)  \sum_{s\geq0} y^{sv} t^s=\sum_{v}  \sigma_{K^0(\P,v)} (\y) \times (1-\y^v t)^{-1}.
\end{eqnarray*}
\end{proof}

The above lemma indeed gives a signed simplicial cone decomposition of $ \mathrm{cone}(\P)$: $\{(1,K(V_v)): v\in V(\P)\}$, where
if $K(\P,v)$ has generators $W(v)=(w_1(v), w_2(v), \dots, w_d(v))$,
 then $V_v$ is the matrix $\left(\begin{matrix}
   W(v)&v\\
   0&1
 \end{matrix}\right)$.

 \begin{cor} \label{Lawrence1991-volume}
Suppose  $\P$ is an integral simple $d$-polytope defined by $\P=\{\alpha | A \alpha \leq b \}$, where $A \in \Z^{m \times d}$ and $b \in \Z^m$.
For each vertex $v$ of $\P$ let $A_v$ be the $d \times d$-matrix composed by the rows of $A$ which are binding at $v$.
Then $A_v$ is invertible. Suppose that $\beta$ is chosen such that neither the entries of $\beta^T\cdot A_v^{-1} $ nor $ \beta^T\cdot  v$ are zero.
 Then
$$\mathrm{vol}(\P)=\frac{1}{d!}\sum\limits_{v}\frac{(-\beta^T\cdot v)^d}{|\det(A_v)| \prod_{j=1}^da_{vj} },$$
where $a_{vj}$ is the $j$-th entry of $\beta^T \cdot A_v^{-1} $.
\end{cor}
\begin{proof}
By assumption, $A_v$ is invertible and the columns of $A_v^{-1}$ are the generators of $K(\P, v)$.

By Lemma \ref{lem-simple}, $\sum_{v} K(V_v)$ is a signed simplicial cone decomposition of $ \mathrm{cone}(\P)$, where $V_v=\left(
           \begin{array}{cc}
            A_v^{-1} & v \\
              0 & 1 \\
           \end{array}
         \right)
$.
It is easy to check $\det(V_v)= \det(A_v^{-1})=(\det(A_v))^{-1}$. So by Theorem \ref{general-full-volcompute}, we obtain
\begin{align*}
\mathrm{vol}(\P)&=\frac{1}{d!}\sum_{v} \Vol_d^\beta(V_v)= \frac{1}{d!}\sum_{v}|\det(A_v)|^{-1}\CT\limits_q \frac{1}{ \prod_{j=1}^d(0-a_{vj}q)\cdot(1-(\beta^T
\cdot v) q)}\\
&=\frac{1}{d!}\sum\limits_{v}\frac{(-\beta^T
\cdot v)^d }{|\det(A_v)| \prod_{j=1}^da_{vj} }.
\end{align*}
Then we get the desired result.
\end{proof}

The two corollaries correspond to two extreme cases. In the former case, for all simplicial cones $V_i$, each of their generators has final entry 1;
 In the latter case, for all simplicial cones $V_i$, each of their generators has final entry 0, but with exactly one exception.

\subsection{A Volume Formula for $d$-Polytope $\P$ in $\R^n$}
In this subsection, we will derive a formula of $\vol(\P)$ for any signed simplicial cone decomposition of cone$(\P)$
$$ \mathrm{cone}(\P)=\sum\limits_{i=1}^N s_i K_i^{v_i} \qquad \textrm{ (short for) } \quad \sigma_{\mathrm{cone}(\P)}(\y;t)=\sum_{i=1}^N  s_i \sigma_{ K_i^{v_i}}(\y,t) ,$$
where $\dim (K_i^{v_i})\leq d+1$.

The situation is different from the full dimensional case. We need to consider the lattice $ \mathrm{span}( \mathrm{cone}(\P))\cap \Z^{n+1}$.
We use the fact $ \mathrm{span}(K_i^{v_i})= \mathrm{span}( \mathrm{cone}(\P))$ if $\dim (K_i^{v_i})=d+1$. When
$\mathrm{cone}(\P)= \{(\alpha^T,s)^T\in \R_{\geq 0}^{n+1}: A\alpha=s\b\}$, its lattice basis can be found
using Smith Normal Form. We need this result later.

\begin{prop}\label{prop-uVol} % thm-uVol in last version 12
Suppose the Smith Normal Form of a $r\times (n+1)$ integer matrix $B$ of rank $r$ is $S=UBV=\big({ diag}(d_1,d_2,\dots,d_r),\mathbf{0}\big)$, where $U$ and $V$ are unimodular matrices. Then the last $n-r+1$ columns of $V$ is a lattice basis of $\{\alpha\in\mathbb{Z}^{n+1}_{\geq0}: B\alpha=0\}$.
\end{prop}
\begin{proof}
  We have
  $$B\alpha=0 \Longleftrightarrow U^{-1} \big({ diag}(d_1,d_2,\dots,d_r),\mathbf{0}\big) V^{-1} \alpha =0\Longleftrightarrow \big({ diag}(d_1,d_2,\dots,d_r),\mathbf{0}\big) V^{-1} \alpha =0.$$
  Now denote by $\y=V^{-1}\alpha$. Then $S \y=0$ if and only if $y_i=0$ for $i=1,\dots r$, so that
  the unit vectors $e_{r+1},e_{r+2},\dots,e_{n+1}$ forms a $\Z$-basis of the null space of $S$. Left multiplied by the unimodular matrix $V$ gives a $\Z$-basis of
 the null space of $B$.
\end{proof}

Let $K$ be a simplicial cone in $\R^n$ with primitive generators $\bar{\mu}_1,\bar{\mu}_2 \dots, \bar{\mu}_d$. We use the notation $$\Det(K)=\Det(\bar{\mu}_1, \dots ,\bar{\mu}_d)=\sqrt{  \det((\bar{\mu}_1,\bar{\mu}_2 \dots, \bar{\mu}_d)^T(\bar{\mu}_1,\bar{\mu}_2 \dots, \bar{\mu}_d))},$$ which reduces to
$|\det((\bar{\mu}_1,\bar{\mu}_2 \dots, \bar{\mu}_d))|$ in the full dimensional case.
\begin{prop}\label{prop-parall}
Let $K^v$ be a simplicial cone as above. Assume $\{\alpha_1,\alpha_2,\dots,\alpha_{d}\}$ is a lattice basis  of $ \mathrm{span}(K^v) \cap \Z^{n}$.
Then the number of lattice points in the fundamental parallelepiped of $K^v$ is equal to
$$\Pi_{K}(\1) =  \frac{\Det(\bar{\mu}_1,\bar{\mu}_2, \dots , \bar{\mu}_{d})}{\Det(\alpha_1,\alpha_2, \dots ,\alpha_{d})}.$$
\end{prop}
\begin{proof}
Since $\{\alpha_1,\alpha_2, \dots ,\alpha_{d}\}$ is an lattice basis of $\mathrm{span}(V) \cap \Z^{n} $, there is an integer matrix $B_{d\times d}$
 such that $(\bar{\mu}_1,\bar{\mu}_2 \dots, \bar{\mu}_d)=(\alpha_1,\alpha_2, \dots ,\alpha_{d})B_{d \times d}$ and
the number of lattice points in the parallelepiped is equal to ${|\det(B)|}.$
Moreover,
\begin{equation*}
\begin{aligned}
 \Det(\bar{\mu}_1,\bar{\mu}_2 \dots, \bar{\mu}_d)=&\sqrt{\det(B^T(\alpha_1, \cdots \alpha_{d})^T(\alpha_1,\alpha_2, \dots ,\alpha_{d})B)}\\
&={|\det(B)|}\cdot \Det(\alpha_1,\alpha_2, \dots ,\alpha_{d}).
\end{aligned}
\end{equation*}
Then the theorem follows.
\end{proof}

Suppose $K^v=K^{v}(\mu_{1}, \mu_{2},\dots,\mu_{h})$, where $\mu_{j}=(\nu_{j}^T,m_{j})^T\in \Q^{n+1}$ and $h\leq d+1$.  Then define
\begin{equation}
  \label{e-vol-simplical-cone}
\widehat{\Vol}_d^\beta  K_i^{v_i} :=\left\{
                      \begin{array}{ll}
                       \frac{\Det(\mu_1,\mu_2, \dots , \mu_{d+1})}{\Det(\alpha_1,\alpha_2, \dots ,\alpha_{d+1})} \CT\limits_q \frac{1}{ \prod_{i=1}^{d+1}(m_{i}-(\beta^T \cdot \nu_{i})  q)}&,  \mbox{if } h=d+1;\\
                           0 &, \mbox{if } h<d+1,
                      \end{array}
                    \right.
\end{equation}
where $\beta$ is admissible for $(\mu_{1}, \mu_{2},\dots,\mu_{h})$ when $\beta^T \cdot \nu_j$ and $m_j$ are not both $0$s for all $j$.

We can get the following theorem about $\vol(\P)$.
\begin{theo}\label{general-volcompute}
Let $\P$ be a rational $d$-polytope in $\R^n$ and $\{\alpha_1,\alpha_2,\cdots,\alpha_{d+1}\}$ is a $\Z$-basis of $ \mathrm{span}( \mathrm{cone}(\P) \cap \Z^{n+1})$.
Suppose $\{(s_i,K_i^{v_i}): 1\leq i\leq N\}$ is a signed simplicial cone decomposition of $ \mathrm{cone}(\P)$, where $rank(K_i^{v_i})\leq d+1$.
Then the relative volume $\mathrm{vol}(\P)$ is equal to
$$\vol(\P)=\frac{1}{d!}\sum\limits_{i=1}^N s_i\widehat{\Vol}_d^\beta   K_i^{v_i},$$
where $\widehat{\Vol}_d^\beta  K_i^{v_i}$ is defined in \eqref{e-vol-simplical-cone}.
\end{theo}
\begin{proof}
The proof is similar to that of Theorem \ref{general-full-volcompute}. For any $K^v=K^v(\mu_1, \mu_2,\dots,\mu_h)$, $\mu_j=(\nu_j^T,m_j)^T\in \Q^{n+1}$ are column vectors, and $(\bar\mu_1, \bar\mu_2,\dots,\bar\mu_h)=(p_1\mu_1, p_2\mu_2,\dots,p_h\mu_h)$, where $p_i\in\Q_{>0}$. We only need to prove
 $\Vol_d^{\beta} \sigma_{K^v}(\y;t)= \widehat{\Vol}_d^\beta K^v$ for $h=d+1$.

By Theorem \ref{prop-parall}, $\sigma_{v+\Pi(K)}(\1;1)$ is equal to
$$\frac{\Det(\bar\mu_1, \cdots ,\bar\mu_{d+1})}{\Det(\alpha_1, \cdots ,\alpha_{d+1})}=\frac{\Det(p_1\mu_{1}, \cdots ,p_{d+1}\mu_{d+1})}{\Det(\alpha_1, \cdots ,\alpha_{d+1})}=\frac{p_1\cdots p_{d+1}\Det(\mu_1, \cdots ,\mu_{d+1})}{\Det(\alpha_1, \cdots ,\alpha_{d+1})}.$$
Using Theorem \ref{theo-vol-formula},  we have
\begin{align*}
\Vol_d^{\beta} F_{K^v}(\y;t)&=\CT_q \frac{\sigma_{v+\Pi(K)}(\1;1)}{ \prod_{k=1}^{d+1}(\bar m_{k}-(\beta^T \cdot \bar\nu_{k})  q)}=\frac{\Det(\bar\mu_1, \cdots ,\bar\mu_{d+1})}{\Det(\alpha_1, \cdots ,\alpha_{d+1})}\CT_q \frac{1}{ \prod_{k=1}^{d+1}(m_{k}-(\beta^T \cdot \nu_{k})  q)}\\
&=\frac{p_1\cdots p_{d+1}\Det(\mu_1, \cdots ,\mu_{d+1})}{\Det(\alpha_1, \cdots ,\alpha_{d+1})}\CT_q \frac{1}{ \prod_{k=1}^{d+1}(p_km_{k}-(\beta^T \cdot p_k\nu_{k})  q)}\\
&=\frac{\Det(\mu_1,\mu_2, \dots , \mu_{d+1})}{\Det(\alpha_1,\alpha_2, \dots ,\alpha_{d+1})} \CT\limits_q \frac{1}{ \prod_{i=1}^{d+1}(m_{i}-(\beta^T \cdot \nu_{i})  q)}=\widehat{\Vol}_d^\beta K^v.
\end{align*}
The theorem then follows.
\end{proof}
\begin{exam}\label{example-pentagon}
Compute the volume of $\P$, where $\P$ is the pentagon depicted in Figure \ref{pentagon1}.
\vspace{-5mm}
\begin{figure}[!ht]
$$
  \hskip .1 in \vcenter{ \includegraphics[height=1.4 in]{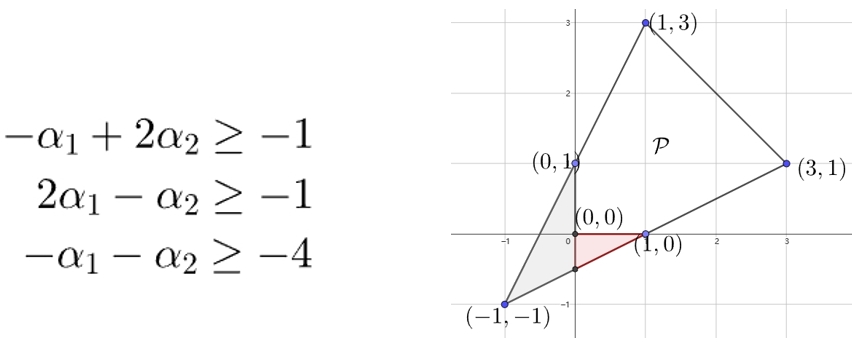}}
$$
\caption{The pentagon $P$ with vertices $(0,0),(1,0),(0,1),(3,1),(1,3)$.
\label{pentagon1}}
\end{figure}
\end{exam}
\vspace{-5mm}

Using MacMahon's partition analysis, we are to compute
$$ Ehr_\P(\y;t)= \Oge \frac{1}{ (1- y_1 \lambda_1^2/(\lambda_2\lambda_3)) (1-y_2 \lambda_2^2/(\lambda_1\lambda_3) )(1-t\lambda_1\lambda_2\lambda_3^4)}.$$
Using existing package \texttt{Ell}, we obtain
\begin{small}
  $${\frac { 12\ monomials }{
 \left( 1-\frac{t}{y_1y_2} \right)  \left( 1-t{y_{{2}}}^{3}y_{{1}}
 \right)  \left(1- t{y_{{1}}}^{3}y_{{2}} \right) }}
{-\frac { 3\ monomials %\left( 1+q+{q}^{2} \right) x_{{1}}q
}{ \left( 1-\frac{{t}^{2}}{y_2}
 \right)  \left(1-\frac{t}{y_1y_2}\right)  \left( 1-ty_{{2}} \right) }
}
 -{\frac {1\ monomial %{q}^{2}
 }{
 \left( 1- \frac{{t}^{2}}{y_2} \right)  \left( 1-ty_{{1}} \right)  \left(1-t \right) }}
 $$
\end{small}
We select $\beta=\mathbf{0}$, then $\vol(\P)=\frac{1}{2}(-\frac{3}{2}+12-\frac{1}{2})=5$.
The decomposition above corresponds to three simplicial cones, which induce a simplex decomposition of $\P$ by intersecting each cones with the hyperplane $s=1$.

The formula given by Theorem \ref{general-volcompute} involves the evaluation of $\Det(\mu_1,\mu_2, \dots , \mu_{d})$, which is a little complicated.
In the next section, we show that this complication can be avoided when $\P$ is defined by \eqref{polytope-A-b}.
\section{A volume formula for the polytope defined by $A\alpha=\b$}\label{sec:simpcone}
Our main result in this section is Theorem \ref{thm-rvol}, which gives a volume formula of $\P$ defined by \eqref{polytope-A-b}.
This formula avoids the computation of determinants and gives a polynomial time algorithm when the dimension is fixed.

We first briefly introduce a new combinatorial simplicial cone decomposition of $ \mathrm{cone}(\P)=\{\alpha\in\mathbb{R}^{n+1}_{\geq0}:(A,-\b)\alpha=0\}$.
Then we apply Lemma \ref{lem-detVdetA} and do further analysis to establish Theorem \ref{thm-rvol}.

\subsection{Introduction to a combinatorial simplicial cone decomposition}
The combinatorial decomposition we are to introduce is inspired by Stanley's work \cite{stanley1974combinatorial} on monster reciprocity theorem. See also \cite{xin2007generalization}.
Stanley's original motivation is to study certain reciprocity property related to the linear system $A\alpha = \mathbf{b}$. For simplicity,
let us consider the cone $K=\{x\in\R^n_{\geq0}:Ax=0, A \in \R^{r\times n}\}$. (For $ \mathrm{cone}(\P)$ we need to replace $A$ by $(A,-\mathbf{b})$.)
 It is of dimension $d=n-r$ if $A$ is of rank $r$ and $K$ contains a \emph{positive vector}, a vector with all positive entries.
We need this technical condition.

In algebraic combinatorics, we have
$$\sigma_{K}(\y)=\CT_{\Lambda} \frac{1}{
\prod_{j=1}^n (1- \lambda^{A_j} y_j)}, \text{where $A_j$ is the $j$-th column  of $A$}.$$
Using constant term extraction, Stanley was able to write $\sigma_K(\y)$ as a sum of groups of similar rational functions involving fractional powers and roots of unity. Such a group was regarded too complicated to deal with, but their similarity was used to derive certain monster reciprocity theorem.
We observe that such a group indeed corresponds to $\sigma_{K_i}$ for certain simplicial cone $K(V_i)$. So the decomposition is of the form $$\sigma_K(\y)=\sum_i s_i \sigma_{K_i}(\y),$$
where $s_i\in \Z$ and each $K_i$ is associated with a nonzero $r\times r$ minor of $A$. Note that in most geometric decompositions, $s_i$ are signs.
The single row case should be known. See, e.g., \cite{xinxudedekindsums}. For a rigourous proof of the general case, see the upcoming paper \cite{conedec}.
%, since the proof is lengthy and digressing.
%\def\hdashline{\hdashline}

We use a slight variation of Stanley's notation. We start with
$$M=\left(\begin{matrix}
         E_n\\
         \hline
         A
        \end{matrix}\right), \text{~where $E_n$ is the identity matrix of order $n$,}
   $$
and successively do \emph{elemental operations} as follows.
Define $M\leftarrow (i,j)$ to be the matrix obtained from $M$ by Gaussian column elimination
with the $(n+i,j)$th entry of $M$ as the \emph{pivot}. We call $M_{n+i,j}\neq 0$ the \emph{pivot item}, and we will ignore row $n+i$ and column $j$
in further operations. More generally, recursively define
$$M\leftarrow ((i_1,\dots, i_k);(j_1,\dots, j_k))= M \leftarrow (i_1,j_1)\leftarrow \cdots \leftarrow (i_k,j_k).$$
Denote the resulting matrix by $\overline{M}_k\langle I;J\rangle$, where we use $\langle I;J\rangle=\langle\{i_1,\dots, i_k\};\{j_1,\dots,i_k\}\rangle $ to emphasis that
rows $n+I$ and columns $J$ of $\overline{M}_k$ are ignored. However, if $k=r$ then $I=[r]=\{1,2,\dots,r\}$, and we treat $\overline{M}_r\langle I;J\rangle:=\overline{M}_r\langle J\rangle$ as rows $n+I$ and columns $J$ removed, and this corresponds to a simplicial cone with $n-r$ columns.

Here we need the fact that for any
permutations $\pi$ and $\pi'$ on $1,2,\dots, k$, we have
 $$M\leftarrow ((i_1,\dots, i_k);(j_1,\dots, j_k))= M\leftarrow ((i_{\pi(1)},\dots, i_{\pi(k)});(j_{\pi'(1)},\dots, j_{\pi'(k)})).$$

Each matrix $M_k$ of the form $\overline{M}_k\langle I;J\rangle$ is indeed a $(n+r-k)\times (n-k)$ matrix with rational entries. It is a \emph{matrix representation}
of an Elliott rational function with fractional powers by
$$\frac{1}{(1-\Lambda^{B_1}\y^{C_1})(1-\Lambda^{B_2}\y^{C_2})\cdots (1-\Lambda^{B_{n-k}}\y^{C_{n-k}})}\longleftrightarrow \left(\begin{matrix}
         C_1&C_2&\cdots&C_{n-k}\\
         \hdashline
         B_1&B_2&\cdots&B_{n-k}
        \end{matrix}\right),$$
where $B_j\in\Q^{r-k}$.
Taking constant terms will be performed on the corresponding matrices.

\def\c{\mathfrak{c}^+}
\def\d{\mathfrak{c}^-}
Consider $\overline{M}_k\langle I;J\rangle$ with columns $\gamma_1,\dots, \gamma_{n-k}$.
The $\gamma_j$ is said to be \emph{small} if its first nonzero entry is positive, and \emph{large} if otherwise.
For each $i\not\in I$, $\gamma_j$ for $j\not\in J$ is said to be
\emph{contributing} (with respect to $i$) if $m_{n+i,j}\gamma_j$ is \emph{small} and \emph{dually contributing} if $m_{n+i,j}\gamma_j$ is \emph{large}, and is called not contributing if $m_{n+i,j}=0$. The number of contributing columns is denoted $\c_i$ and the number of dually contributing columns is denoted $\d_i$.

\def\sgn{\mathrm{sgn}}
%Define
%$$\CT_{i,j} \bar M_k<I,J>=  \epsilon \bar M_k<I,J> \leftarrow (i,j) =\epsilon \bar M_{k+1}<(I,i);(J,j)>,$$
%where if $\gamma_j$ is not contributing then $\epsilon=0$;
%elif $\gamma_j$ is contributing then $\epsilon=\sgn(m_{n+i,j})$;
%else $\gamma_j$ is dually contributing, then  $\epsilon=-\sgn(m_{n+i,j})$.

A translation of the constant term extraction in \cite{xin2007generalization} suggests the following definition.
\begin{align*}
  \CT_i\overline{M}_k\langle I;J\rangle &=\sum_{\substack{j: \gamma_j
                   \text{ is contributing}}} \sgn(m_{n+i,j})  \overline{M}_k\langle I\cup\{i\};J\cup\{j\}\rangle\\
                   &= \sum_{\substack{j:\gamma_j
                   \text{ is dually contributing}}}-\sgn(m_{n+i,j})  \overline{M}_k\langle I\cup\{i\};J\cup\{j\}\rangle.
\end{align*}
An obvious choice is to use the first formula if $\c_i\leq \d_i$ and use the second one if otherwise.

Now we are able to describe the new algorithm in \cite{conedec}.
\begin{algor}[\texttt{SimpCone}]\label{algor-Simpcone}
\ \\
Input: A matrix $A\in \Z^{r\times n}$ such that the cone $K=\{x \ge 0: \ Ax=0\}$ has at least one positive solution. \\
Output: A signed simplicial cone decomposition of $K$.
\begin{itemize}
 \item[1]  Construct the set $\{(1, O_{0,1}, \varnothing, \varnothing )\}$ where $ O_{0,1}=M$.

 \item[2] For $i$ from $0$ to $r-1$, compute $O_{i+1}$ from $O_
 {i}$  as follows:

 \item[2.1] For each $(s_{i,j}, O_{i,j}, I_{i,j}, J_{i,j})$ of $O_i$, compute $\c_k(O_{i,j}\langle I_{i,j}; J_{i,j}\rangle)$
 and $\d_k(O_{i,j}\langle I_{i,j}; J_{i,j}\rangle)$ for each $k\not\in I_{i,j}$. Suppose the minimum is attained at $k_0$.

 \item[2.2] Compute $s_{i,j}\CT_{k_0} O_{i,j}$ with respect to $I_{i,j}$ and $J_{i,j}$,
 and collect the terms into $O_ {i+1}$.

 \item[3] Output $O_r=\{(s_1, O_{r,1},J_{r,1}),\dots,(s_N, O_{r,N}, J_{r,N}\}$.
 \end{itemize}
\end{algor}

\begin{rem}\label{rem-Simpcone}
The condition $K$ has at least one positive solution cannot be dropped. It guarantees that all Elliott-rational functions $F$ involved
are proper in $\lambda_i$ and $F|_{\lambda_i=0}=0.$
\end{rem}

\begin{rem}
  The choice in Step 2.1 does not guarantee an optimal decomposition. It is also possible to use the fact
 $K=\{x \ge 0: \ U Ax=0\}$ for any nonsingular matrix $U$.
\end{rem}

The next example illustrates a cancellation in our weighted cone decomposition.
\begin{exam}\label{exam-2-5}
Take $A=  \left( \begin {array}{ccccc} 5&6&-1&-2&-3\\  -1&1&-3
&1&1\end {array} \right)
$ as input. Then \texttt{SimpCone} works as follows.
%\begin{itemize}
 \item[(1)]  Construct $O_0={O_{01}}:$
   $$
     O_{0,1}=  \left( \begin {array}{ccccc} 1&0&0&0&0\\  0&1&0&0&0
\\  0&0&1&0&0\\  0&0&0&1&0
\\  0&0&0&0&1\\
 \hdashline%[0.8pt/1pt]
   5&6&-1&-2&-3\\
   -1&1&-3&1&1
\end {array} \right)
   $$
 \item [(2.1)] Compute $\CT_1 O_{0,1}$. Now row 6 has 2 contributing columns, so
\begin{align*}
\CT_1 O_{0,1}= O_{1,1}+O_{1,2}&=(O_{0,1})_1\langle \{1\},\{1\}\rangle+(O_{0,1})_1\langle \{1\},\{2\}\rangle\\
&=\left( \begin {array}{ccccc}  \color{red}1&-{\frac{6}{5}}&{\frac{1}{5}}&{\frac{2}
{5}}&{\frac{3}{5}}\\  \color{red}0&1&0&0&0\\  \color{red}0&0&1&0&0\\ \color{red}0&0&0&1&0\\  \color{red}0&0&0&0&1\\
 \hdashline%[0.8pt/1pt]
 \color{red}\textcircled{5}&\color{red}0&\color{red}0&\color{red}0&\color{red}0\\  \color{red}-1&{\frac{11}{5}}&-
{\frac{16}{5}}&{\frac{3}{5}}&{\frac{2}{5}}\end {array} \right)+ \left( \begin {array}{ccccc} 1&\color{red}0&0&0&0\\  -{\frac{5}
{6}}&\color{red}1&{\frac{1}{6}}&{\frac{1}{3}}&{\frac{1}{2}}\\  0
&\color{red}0&1&0&0\\  0&\color{red}0&0&1&0\\  0&\color{red}0&0&0&1
\\
 \hdashline%[0.8pt/1pt]
  \color{red}0&\color{red}\textcircled{6}&\color{red}0&\color{red}0&\color{red}0\\  -{\frac{11}{6}}&\color{red}1&-
{\frac{17}{6}}&{\frac{4}{3}}&{\frac{3}{2}}\end {array} \right).
\end{align*}
Then
$O_1=\{(1,O_{1,1},\{1\};\{1\}), (1, O_{1,2},  \{1\},\{2\} )\}$.
\item[(2.2)] Compute $O_2$.
\begin{align*}
&\CT_2O_{1,1}=-(O_{0,1})_2\langle\{1,2\};\{1,2\}\rangle+(O_{0,1})_2\langle\{1,2\};\{1,3\}\rangle,\\
&\CT_2O_{1,2}=(O_{0,1})_2\langle\{1,2\},\{2,1\}\rangle+(O_{0,1})_2\langle\{1,2\};\{2,3\}\rangle.
\end{align*}
Then collect into $O_2$ (note the cancellation here),
$$\CT_2O_{1,1}+\CT_2O_{1,2}=O_{2,1}+O_{2,2},$$
where $$O_{2,1}=(O_{0,1})_2\langle \{1,3\}\rangle=
 \left( \begin {array}{ccc} -{\frac{17}{16}}&{\frac{7}{16}}&{\frac{5}{
8}}\\  1&0&0\\  {\frac{11}{16}}&{
\frac{3}{16}}&{\frac{1}{8}}\\  0&1&0
\\  0&0&1\end {array} \right)
,~O_{2,2}=(O_{0,1})_2\langle\{2,3\}\rangle= \left( \begin {array}{ccc} 1&0&0\\  -{\frac{16}{17}}
&{\frac{7}{17}}&{\frac{10}{17}}\\  -{\frac{11}{17}}&{
\frac{8}{17}}&{\frac{9}{17}}\\  0&1&0
\\  0&0&1\end {array} \right).$$
Then output $O_2=\{(1, O_{2,1},\{1,3\}),(1, O_{2,2},\{2,3\})\}$.
% \end{itemize}
 \end{exam}

The complexity of Algorithm Simpcone is hard to analyze. An obvious upper bound on the number of cones is $\binom{n}{r}$ since the
number of different $J_{r,i}$ is at most $\binom{n}{r}$. Another obvious bound is $(n/2)^r$, since in each step we have a choice.

We were called attention by Matthias Beck on the Algorithm Polyhedral Omega in \cite{breuer2017polyhedral}.
That algorithm is similar to our \texttt{Simpcone} but quite different in methods and outcomes. They are interested in all non-negative integer solutions to the system
$A'\y\geq \b$. The algorithm employs a signed simplicial cone decomposition  $\{(s_i,K_i^{v_i})\}$ of the set  $\{\y\in\R^d: A'\y\geq \b\}$  with complexity
$ O(s^2 \binom{d+r-1}{d} (d+r)^3 d^2 r^2 \log(d)^2 \log(d+r)^2)$, where $s$ is the maximum input size of the entries of $A'$ and $\b$.
We can obtain a similar complexity result for  \texttt{Simpcone}  as $ O(s^2 \binom{n}{r} (n)^3 (n-r)^2 r^2 \log(n-r)^2 \log(n)^2)$.

\subsection{Improvement of Volume Formula by \texttt{Simpcone}}
Algorithm \ref{algor-Simpcone} \texttt{Simpcone} gives a weighted simplicial cone decomposition $\sum_i s_i K(V_i)$ of $ \mathrm{cone}(\P)$ with $V_i\in \Q^{n\times (d+1)}$ and $s_i\in\Z$. Such a formula can be used by Theorem \ref{general-volcompute} to compute
$\vol(\P)$. Here we derive a nicer formula by introducing a seemingly new invariant as follows.

\begin{defn}
Let $M$ be a $(n+r)\times n$ matrix partitioned as $\begin{pmatrix}
         C \\
          \hdashline
         B
       \end{pmatrix}
       =\left(\begin{array}{c:c}
    C_1 &C_2  \\
    \hdashline
    B_1 &B_2
  \end{array}\right)$, where $C\in \R^{n\times n}$ is invertible
 and $B_1\in \R^{r\times r}$. If $B_1$ is invertible, then define $\mathcal{I}(M):=\Det(B_1)\Det(C_2-C_1B_1^{-1}B_2)$. It reduces to
 $\mathcal{I}(M)=\Det(B_1)\Det(C_2)$ when $B_2$ is zero.
\end{defn}
We will perform elementary operations $\tau$ on $M$. If $\mathcal{I}(\tau(M))=\mathcal{I}(M)$, then we will say that $\mathcal{I}(M)$ is \emph{invariant} under $\tau$, or
$\tau$ is an \emph{invariant operation} over $M$. Similarly, we can say $\tau_m\cdots \tau_2\tau_1$ is an invariant operation, or
$M\mapsto PMQ$ is invariant for suitable square matrices $P$ and $Q$.

\begin{prop}\label{prop-I-invariant}
The $\mathcal{I}(M)$ is invariant under the elementary row operations on $B$ and elementary column operations on $M$ of the following three types:
i) $\mathcal{E}_i$ for replacing the $i$-th column by its negative;
ii) $\mathcal{E}_{i,j}$ for exchanging the $i$-th column with the $j$-th column;
iii) $\mathcal{A}_{i,j,c}$ for adding the $c$ multiple of the $i$-th column to the $j$-th column.
\end{prop}
\begin{proof}
For convenience, we let $P$ be a matrix in $\R^{r\times r}$ with $\det(P)=\pm 1$, $Q$ be a matrix in $\R^{n-r\times n-r}$ with $\det(Q)=\pm 1$.

The invariance for row operations is simple. Assume $B$ is transformed to $PB=(PB_1, PB_2)$. Then
 we have $$\mathcal{I}(M')%=\mathcal{I}(\begin{pmatrix}
       %  C_1 &C_2\\
%         PB_1 &PB_2
%       \end{pmatrix} )
       = \Det(PB_1)\Det(C_2-C_1(PB_1)^{-1}(PB_2))=\Det(B_1)\Det(C_2-C_1B_1^{-1}B_2)=\mathcal{I}(M).$$

For column operations, we need to consider the following $5$ cases.

\begin{itemize}
\item [1.] Column operations among the first $r$ columns.

\item [2.] Column operations among the last $n-r$ columns.

\item [3.] $\mathcal{A}_{i,j,c}$ for $i\le r$ and $j>r$.

\item [4.] $\mathcal{E}_{i,j}$ for $i\le r$ and $j>r$.

\item [5.] $\mathcal{A}_{j,i,c}$ for $i\le r$ and $j>r$.
\end{itemize}
The first two cases can be combined as
$\left(\begin{array}{c:c}
         C_1 &C_2\\
         \hdashline
         B_1 &B_2
       \end{array}\right) \longrightarrow  \left(\begin{array}{c:c}
         C_1P &C_2Q \\
          \hdashline
         B_1P &B_2Q
       \end{array}\right)$.
       The third case corresponds to  $\left(\begin{array}{c:c}
         C_1 &C_2\\
         \hdashline
         B_1 &B_2
      \end{array}\right) \longrightarrow   \left(\begin{array}{c:c}
        C_1& C_2+C_1P' \\
        \hdashline
        B_1&B_2+B_1 P'
       \end{array}\right)$.
The proof for these cases are similar to the proof for row operations on $B$.

For Case 4, by the first three cases, we may assume $i=r$ and recall $j>r$.
Since row operations commute with column operations,
we may use suitable row operations to simplify $B_1$. So we assume $B_1=diag(p_1,\dots, p_r)$, where $p_k\neq 0$.
Next observe that for $i'<r$, $\mathcal{A}_{i',r,c}, \mathcal{A}_{i',j,c}$ are invariant operations and $\mathcal{E}_{r,j}\mathcal{A}_{i',j,c}=\mathcal{A}_{i',r,c}\mathcal{E}_{r,j}$,
we may further assume the first $r-1$ rows of $B_2$ are all zeros.

Now for $\mathcal{E}_{r,j}$ to be invariant, we need $(B_2)_{r,j-r}\neq 0$. Without loss of generality, assume $j=n$ and $(B_2)_{r,n-r}:=b_n\neq 0$.
By the equality $\mathcal{A}_{n,j,c}\mathcal{E}_{r,n}=\mathcal{E}_{r,n}\mathcal{A}_{r,j,c}$ for $r<j<n$, and using suitable $c$ for each $j$, we may further
assume $(B_2)_{r,j-r}=0$ for all $j$ except $(B_2)_{r,n-r}\neq 0$. Now $M$ is of the form
$$ M=\left(\begin{array}{cc|cc}
    C_1'&\gamma_r &C_2'&\gamma_n\\
    \hdashline
    B_1'&\mathbf{0} &\mathbf{0}&\mathbf{0} \\
    0   &  p_r  & 0 &b_n
  \end{array}\right)  \longrightarrow \left(\begin{array}{cc:cc}
    C_1'&\gamma_r &C_2'&\gamma_n-\gamma_r p_r^{-1}b_n \\
    \hdashline
    B_1'&\mathbf{0} &\mathbf{0}&\mathbf{0} \\
    0   &  p_r  & 0 &0  \end{array}\right),$$
where $B_1'=diag(p_1,\dots,p_{r-1})$.
For $M'=\mathcal{E}_{r,n}M$, we have
$$ M'=\left(\begin{array}{cc|cc}
    C_1'&\gamma_n &C_2'&\gamma_r\\
     \hdashline
    B_1'&\mathbf{0} &\mathbf{0}&\mathbf{0} \\
    0   &  b_n  & 0 &p_r
  \end{array}\right)  \longrightarrow  \left(\begin{array}{cc:cc}
    C_1'&\gamma_r &C_2'&\gamma_r-\gamma_n b_n^{-1}p_r \\
     \hdashline
    B_1'&\mathbf{0} &\mathbf{0}&\mathbf{0} \\
    0   &  b_n  & 0 &0   \end{array}\right).$$

Then we have
\begin{align*}
\mathcal{I}(M')&=|p_1\cdots p_{r-1}b_n|\cdot\Det\left(C_2',\gamma_r-\gamma_n b_n^{-1}p_r\right)\\
&=|p_1\cdots p_{r-1}|\cdot\Det\left(C_2',\gamma_rb_n-\gamma_n p_r\right)\\
&=|p_1\cdots p_r|\cdot\Det\left(C_2',\gamma_n-\gamma_r p_r^{-1}b_n\right)=\mathcal{I}(M).
\end{align*}
For Case 5, we show that $M'=\mathcal{A}_{j,i,c}M$, where $1\leq i\leq r<j\leq n$, can be obtained through several invariant operations, namely,
$M'=\mathcal{E}_{i,j} \mathcal{A}_{i,j,c} \mathcal{E}_{i,j} M$. To see this, we focus on the $i,j$-th columns, say $\alpha_i, \alpha_j$, so that
$$(\alpha_i,\alpha_j)\longrightarrow (\alpha_j,\alpha_i) \longrightarrow (\alpha_j,\alpha_i +c\alpha_j)  \longrightarrow (\alpha_i +c\alpha_j, \alpha_j).$$
The proposition then follows.
\end{proof}

\begin{lem}
  \label{lem-detVdetA}
Suppose $\alpha_1, \cdots ,\alpha_{d+1}$ is a $\Z$-basis of $\{\alpha\in\mathbb{R}^{n+1}_{\geq 0}:B\alpha=0\}$ and Smith Normal Form of $B$ is  $S= \big({ diag}(d_1,d_2,\dots,d_r),\mathbf{0}\big)$.
\texttt{Simpcone} acts on $B\alpha=0$ gives a weighted simplicial cone decomposition $\sum s_iK(V_i)$.
Assume $p_{i,1},p_{i,2},...,p_{i,r}$ are the pivot items in the process of obtaining $V_i$.
Then we have
$$\frac{\Det(V_i)}{\Det(\alpha_1,\cdots,\alpha_{d+1})}=\left| \frac{d_1d_2\cdots d_r}{p_{i,1}p_{i,2}\cdots p_{i,r}}\right|.$$
\end{lem}
We remark that the $p_{i,1}p_{i,2}\cdots p_{i,r}$ is a $r\times r$ minor, and $d_1d_2\cdots d_r$ is known to be the $\gcd$ of all
the $r\times r$ minors.
\begin{proof}
Partition the  matrix $M$ as $\begin{pmatrix}
         E_n \\
         \hdashline
         B
       \end{pmatrix}=\left(\begin{array}{c:c}
         C_1 &C_2\\
         \hdashline
         B_1 &B_2
       \end{array}\right)$, where $C_1=(E_r,\mathbf{0})^T,~C_2=(\mathbf{0},E_{n-r+1})^T$.
Without loss of generality, we may assume $V_i$ is obtained from $M_r\langle [r],[r]\rangle $. Then $\mathcal{I}(M)=\Det(B_1)\Det(C_2-C_1B_1^{-1}B_2)=|p_{i,1}p_{i,2}\cdots p_{i,r}|\Det(V_i)$.

Suppose the Smith Normal Form of $B$ is $S=UBV$.  Then $M$ is transformed to $M'=\begin{pmatrix}
    V \\
    \hdashline
    UBV
  \end{pmatrix}=\left(\begin{array}{c:c}
         V' &V''\\
         \hdashline
        D &\mathbf{0}
       \end{array}\right)$, where $D=diag(d_1,\dots,d_r)$. By Proposition \ref{prop-uVol}, we have $\Det(V'')=\Det(\alpha_1,\cdots,\alpha_{d+1})$. Then
$\mathcal{I}(M')=|d_1d_2\cdots d_r|\Det(\alpha_1,\cdots,\alpha_{d+1})$.

By Proposition \ref{prop-I-invariant}, we have $\mathcal{I}(M)=\mathcal{I}(M')$, that is $$|p_{i,1}p_{i,2}\cdots p_{i,r}|\Det(V_i)=|d_1d_2\cdots d_r|\Det(\alpha_1,\cdots,\alpha_{d+1}).$$
This completes the proof.
\end{proof}

Let $K(V)$ be a simplicial cone, where $V=(\mu_{1}, \mu_{2},\dots,\mu_{d+1})$ with $\mu_{j}=(\nu_{j}^T,m_{j})^T$ and the pivot items be $p_{1},p_{2},...,p_{r}$ in the process of obtaining $V$ by using \texttt{Simpcone}. From the above lemma, \eqref{e-vol-simplical-cone} is equivalent to
\begin{equation}\label{Vol-beta-Simpcone}
\widehat{\Vol}_d^\beta  K(V)=\left| \frac{d_1d_2\cdots d_r}{p_{1}p_{2}\cdots p_{r}}\right| \CT\limits_q \frac{1}{ \prod_{i=1}^{d+1}(m_{i}-(\beta^T \cdot \nu_{i})  q)},
\end{equation}
where $\beta$ is admissible for  $V=(\mu_{1}, \mu_{2},\dots,\mu_{d+1})$ when $\beta^T \cdot \nu_j$ and $m_j$ are not both $0$s for all $j$.

Applying Theorem \ref{general-volcompute} and Lemma \ref{lem-detVdetA}, we can get the following result about $\mathrm{vol}(\P)$.
\begin{theo}\label{thm-rvol}
Let $B=\left(A,-\b\right)$ with Smith Normal Form
 $S= \big({ diag}(d_1,d_2,\dots,d_r),\mathbf{0}\big)$.
Suppose \texttt{Simpcone} acts on $B\alpha=0$ and gives a weighted simplicial cone decomposition $\sum_i s_iK(V_i)$ with $rank(V_i)=d+1$.  The pivot items are $p_{i,1},p_{i,2},...,p_{i,r}$ when we get $V_i$.
Then the relative volume $\mathrm{vol}(\P)$ is equal to
$$\mathrm{vol}(\P)=\frac{1}{d!}\sum\limits_i  s_i \widehat{\Vol}_d^\beta  K(V_i), $$
where $\widehat{\Vol}_d^\beta  K(V_i)$ is defined in \eqref{Vol-beta-Simpcone}.
\end{theo}

\begin{exam}
Suppose $\P=\{\alpha\in\R^5_{\ge 0} : A\alpha =\b\}$, where $A=\left( \begin {array}{ccccc} 2&3&-1&-1&0\\  -1&3&0
&1&0\\  7&0&0&0&1\end {array} \right)
,~\b=\left( \begin {array}{c} -3\\  2
\\  3\end {array} \right).
$ Then $ \mathrm{cone}(\P)=\{(\alpha^T,s)^T\in\R^6_{\ge 0} : B(\alpha^T,s)^T=\mathbf{0}\}$, where $B=(A,-\b)$. The weighted simplicial cone decomposition of $ \mathrm{cone}(\P)$ is $K(V_1)+K(V_2)$, where $$V_1= \left( \begin {array}{ccc} 1&0&0\\  {{17}/{9}}&-
{{1}/{3}}&{{2}/{9}}\\  {{44}/{3}}&-2&{
\frac{5}{3}}\\  0&1&0\\  0&0&1
\\  {{7}/{3}}&0&{{1}/{3}}\end {array}
 \right)
,~V_2= \left( \begin {array}{ccc} 1&0&0\\  0&1&0
\\  {{10}/{3}}&6&{{1}/{3}}
\\  {{17}/{3}}&-3&{{2}/{3}}
\\  0&0&1\\  {{7}/{3}}&0&{{1}/{3}}\end {array} \right).$$

Let $p_{i,j}$ be corresponding pivots item  of $V_i$ as described in Theorem \ref{thm-rvol}, then we have
$$|p_{1,1}p_{1,2}p_{1,3}|=9, ~|p_{2,1}p_{2,2}p_{2,3}|=3.$$
We can also easily get $|d_1d_2d_3|=1$.
Then we select $\beta=(15,-22,-7,10,-20)^T$ and by Theorem \ref{thm-rvol},  $vol(\P)=\frac{1}{2!}(\frac{1}{9}\cdot \frac{15597}{2303}-\frac{1}{3}\cdot \frac{823}{2303})=\frac{31}{98} $.
\end{exam}

%\begin{exam}
%Suppose $\P=\{\alpha\in\R^n_{\ge 0} : \sum\limits_{i=1}^n a_i\alpha_i=b\}$, where $\gcd(a_1,\dots, a_n)=1$.
%$$E_1=\frac{1}{(1-\lambda^{a_1}y_1)(1-\lambda^{a_2}y_2)\cdots(1-\lambda^{a_n}y_n)(1-t\lambda^{-b})}\longleftrightarrow
%M=\begin{pmatrix}
% 1& & & &\\
%  &1& & &\\
%  & &\ddots& &\\
%  & & & 1&\\
%   & & & &1\\
%  \hdashline
% a_1&a_2&\cdots&a_n&-b
%\end{pmatrix}.$$
%
%$$E_2=\CT\limits_{\lambda}E_1=\frac{1}{(1-t^{a_1}y_1)(1-t^{a_2}y_2)\cdots(1-t^{a_n}y_n)}\longleftrightarrow
%M\langle\{1\};\{n+1\}\rangle=\begin{pmatrix}
% 1& & & \\
%  &1& & \\
%  & &\ddots& \\
%  & & & 1\\
% \frac{a_1}b& \frac{a_2}b&\cdots& \frac{a_n}b
%\end{pmatrix}$$
%We observe that the pivot item is $p_1=-b$ and $d_1=1$.  Then we select $\beta=\mathbf{0}$ and by Theorem \ref{thm-rvol},  $\vol(\P)=\frac{b^{n-1}}{(n-1)!a_1\cdots a_n}.$
%
%Note that using Brion's theorem will  produce $n$ terms.
%\end{exam}

\section{Primal-Dual method for volume computations}\label{sec:primal-dual}
Brion's  polarization trick \cite{brion1988points} established a close connection between a cone $K$ in $\R^d$ and its dual cone $K^*$
defined by $K^*:=\{\alpha \in \R^d : \langle \alpha, \beta \rangle \ge 0 \text{ for each } \beta \in K\}$.
If $K$ is not full-dimensional, then $K^*$ contains a straight line and $\sigma_{K^*}(\y)=0$. If $K$ is full-dimensional,
then $K^*$ is also full-dimensional. Furthermore, $\sigma_{K^*}=\sum_{i} \epsilon_i \sigma_{K_i^*} \Leftrightarrow \sigma_{K}=\sum_{i} \epsilon_i \sigma_{K_i}$. See \cite{barvinokPommersheim1999summary} for detail.

Our new volume formulas in Theorems \ref{general-full-volcompute} allows us to compute volume in either the primal space or the dual space.
What we need is a decomposition of $\sigma_{\cone{(\P)}}(\y)$, which is equivalent to its dual decomposition by Brion's polarization trick.
Thus we can use the decompositions either in the primal space or in the dual space. Of course we shall use the better one.

The pseudo code of our Primal-Dual volume computation algorithm for $\vol(\P)$ can be described as follows.
\begin{itemize}
  \item Estimate the number of vertices $nv$ of $\P$ and the number of facets $nf$ of $\P$.

  \item If $nv\le nf$, then triangulation $\cone(\P)$; Otherwise triangulate $\cone(\P)^*$ and dual back. In either case, we obtain a decomposition of
  $\sigma_{\cone{\P}}(\y)$.

  \item Apply Theorems \ref{general-full-volcompute}.
\end{itemize}
We don't have an implementation of this algorithm yet.

To carry out the idea, we need an efficient decomposition of $\cone(\P)$ into simplicial cones.
 Algorithms have been developed by geometers to efficiently decompose a full dimensional cone into signed simplicial cones.
See \cite{sack1999handbook} and \cite{lee1997subdivisions}. As such algorithms were not available to us in the first version of the paper,
we developed the \texttt{Simpcone} algorithm. Now we have a C++ implementation of \texttt{SimpCone}. In the next section, we will see that it has a good performance in our computer experiment.
And it has much room for improvement.

Currently, we use Delaunay triangulation (DT for short), which is a good triangulation method for polytopes.
In \cite{de1995triangulations},  De Loera describes the DT of cones, which is in turn based on Lee \cite{lee1991reg}, but
apply it to cones instead of polytopes. It has been implemented in \texttt{LattE}, which is a nice C++ package for lattice point counting problems \cite{de2004effectiveLattE}. The package provides the DT in different modules like \texttt{cdd} and \texttt{4ti2}.

Computing the volume in the dual space seems completely new because of the following reasons.
\begin{itemize}
\item The last coordinate $y_{d+1}$ in $\cone(\P)$ plays a very special role in the literature. It was observed that applying DT to $\cone(\P)$
 is equivalent to triangulating $\P$ itself. See \cite[Section 5]{verdoolaege2005computation} for detail.

%\item We only need a decomposition of $\sigma_{\cone{\P}}(\y)$, where each coordinate plays the same role. Thus we can treat $\cone(\P)$
%as $\cone(\P')$ and apply DT, where $\P'$ is the intersection of $ \cone(\P)$ with $y_i=1$ for any particular $i$.

\item If we apply DT in the dual space and dual back, the resulting simplicial cone decomposition may corresponds to polyhedron decomposition of $\P$, rather than simplices decomposition of $\P$. This phenomenon first appears in our Simpcone algorithm.
\end{itemize}

We give a concrete example of the perturbed square $\P$ depicted on the left of Figure \ref{2-cube-dec}.
For simplicity, we make the top line segment
horizontal. Our Simpcone algorithm
gives $\sigma_{\mathrm{cone}(\P)}(\y;t)= \sigma_{K(V_1)} (\y) + \sigma_{K(V_2)} (\y) $, where
$$
V_1= \left( \begin {array}{ccc}
1&6&-1\\
1&0&0\\
1&5&0\end {array} \right)
 \text{~and~}
V_2=\left( \begin {array}{ccc} 1&0&0\\0&1&0
\\0&1&1\end {array} \right).
$$
This corresponds to decomposing $\P$ into the two polyhedrons 
$K(V_1) \cap (t=1)$ and  $K(V_2) \cap (t=1)$ depicted on the right of Figure \ref{2-cube-dec}.

We conclude this subsection by describing several volume algorithms accessible to us for comparison.
We mainly discuss the model $\cone(\P) = \{ \alpha \in \R_{\geq 0}^{n+1}: (A,-\b) \alpha =0\}$, where $A_{r\times n}$
and $\b$ is integral. Then $\cone(\P)$ is a $d+1=n-r+1$ dimensional cone defined by $n+1$  hyperplanes. This implies that the dual cone $\cone(\P)^*$
has at most $n+1$ extreme rays. On the other hand, $\cone(\P)$ may have too many extreme rays in many situations. Thus this model is usually
better solved using dual cone decompositions.

\textbf{Primal DT}: Convert $\cone(\P)$ into a full dimensional cone, apply DT in the primal space, and compute the volume separately by Theorem \ref{general-full-volcompute}. This is equivalent to simplex decomposition.

\textbf{LV Dec}: For each vertex, compute the tangent cone, apply DT, and compute the volume separately by Corollary \ref{Lawrence1991-volume}.

In \cite{bueler2000exact}, the authors survey various triangulation algorithms and LV Dev, and propose an improvement of the latter.
They also present a hybrid approach called HOT for ``hybrid orthonormalisation technique". They implement these algorithms in the software package \texttt{vinci}.
These methods seem to have a high complexity when the polytope $\P$ has too many vertices.

The next two algorithms use our new volume formula and decompose in the dual space.

\textbf{Dual DT}: Convert $\cone(\P)$ into a full dimensional cone, compute its dual, apply DT (in the dual space),
dual back, and then compute the volume separately by Theorem \ref{general-full-volcompute}.

\textbf{SimpCone}: Use Simpcone to decompose, and compute the volume separately by Theorem \ref{thm-rvol}.

\section{Computer experiment}\label{sec:compexp}
Basically, we report the number of cones obtained by the described algorithms.
\subsection{On the perturbed $n$-cube}
Let $C_n$ be the $n$ dimensional cube, i.e., defined by $\{\alpha \in \R_{\ge 0}^n:  e_i \cdot \alpha \leq 1,\ 1\leq i \leq n\}$.
In our notation, $\cone(C_n)$ is specified by the matrix $(I_n,I_n,-\mathbf{1})$.
We perturb $C_n$ on its normal vectors $e_i$. That is, we replace the $e_i$ with $e_i-\beta_i$, where the entries of $\beta_i$ are small enough. The resulting
polytope is called the perturbed $n$-cube, and denoted $\P_n$. Then $\cone(\P_n)$ is specified by the matrix
$(I_n-B,I_n,-\mathbf{1})$, where $B$, having sufficiently small entries, is referred as the \emph{perturbing matrix}.
We are interested with the number of cones to decompose $\cone(\P_n)$.

In geometrical decompositions, it suffices to assume $\P_n=C_n$.

The simplexity of $C_n$ is the minimal number of simplices in simplicial dissections of $C_n$.
It has received much attention. See, e.g., \cite{glazyrin2012lower}, where the authors obtained a new asymptotic lower bound $(n+1)^{\frac{n-1}{2}} $ on this simplexity. This huge bound suggests that simplex decomposition is usually not good for volume computation.
Indeed, for $\cone(C_8)$, Primal DT gives  $24135$ simplicial cones.

It is well-known that $C_n$ is a simple polytope with $2^n$ vertices. Thus using LV Dec only produce $2^n$ simplicial cones.
This is great success when comparing with simplex decomposition.
In \cite{bueler2000exact}, the authors compare several volume algorithms based on different triangulation algorithms, and conclude that LV Dec has the fewest terms in their expressions.

When we decompose $\cone(\P_n)$, it is only similar to that of $\cone(C_n)$ when the perturbed matrix $B$ has
sufficiently small entries. In that case, one usually obtain $2^{n-1}$ (which is $2^7=128$ when $n=8$) simplicial cones by \texttt{Dual DT} or the \texttt{SimpCone} algorithm.

%\begin{multline*}
%\sigma_{\mathrm{cone}(\P)}(\y;t)=\Oge\frac{1}{(1-\lambda^{-1}y_1)(1-\mu^{-1}y_2)(1-\lambda\mu t)} =\sigma_{K(V_1)} (\y) + \sigma_{K(V_2)} (\y)\\
%=\frac {  {t}^{4}{x}^{4}+{t}^{3}{x}^{3}+{t}^{2}{x}^{2}+tx+1
% }{ \left( 1-x^{-1} \right)  \left( 1-{t}^{5}{x}^{6} \right)
% \left(1- txy \right) }
% +\frac{1}{(1-y_2t)(1-t)(1-y_1)}.
%\end{multline*}

%
%\begin{multline*}
%\sigma_{\mathrm{cone}(\P)}(\y;t)=\Oge\frac{1}{(1-\lambda^{-1}y_1)(1-\mu^{-1}y_2)(1-\lambda\mu t)}\\
%=\frac{1}{(1-y_1y_2t)(1-y_1t)(1-y_1^{-1})}+\frac{1}{(1-y_2t)(1-t)(1-y_1)}.
%\end{multline*}

\subsection{On Birkhoff polytopes and Magic square polytopes}
The Birkhoff polytope $\mathcal{B}_n$ of order $n$ is defined by the following linear constraints:
\begin{align*}
  a_{i,1}+a_{i,2}+\cdots+a_{i,n}=1,  & \text{ for }1\le i\le n \\
  a_{1,j}+a_{2,j}+\cdots+a_{n,j}=1,  &  \text{ for }1\le j\le n \\
  a_{i,j}\geq 0,  &\text{ for }1\le i,j\le n.
\end{align*}
Elements in $\mathcal{B}_n$ are also called the doubly stochastic matrices of order $n$.
The magic square polytope $\mathcal{MS}_n$ is the intersection of $\mathcal{B}_n$ with the two additional linear constraints (hyper planes) $a_{1,1}+a_{2,2}+\cdots +a_{n,n}=1, \ a_{n,1}+a_{n-1,2}+\cdots +a_{1,n}=1$ corresponding to the diagonals. Lattice points in $s\mathcal{MS}_n$ are called magic squares with magic sum $s$.

The volumes of the Birkhoff polytopes $\mathcal{B}_n$ was only known up to $n=10$, with the record kept by Beck and Pixton \cite{beck2003ehrhart}
using residue computation. Their method is not applicable to general polytopes. The Ehrhart series of the magic square polytope $\mathcal{MS}_6$ was computed by Xin \cite{xin2015euclid}. It can be used to derive $\vol \mathcal{MS}_6$.

In Tables \ref{table-Birkhoof} and \ref{table-Magic}, we recompute the volumes of these two polytopes using Theorem \ref{thm-rvol}.
We also tried all the algorithms in the package \texttt{Vinci} on  $\mathcal{B}_n$. LV Dec is the only algorithm succeeds for the volume of $\mathcal{B}_6$.
In Table \ref{table-Dec}, we compare the number of simplicial cones produced by \texttt{SimpCone}, LV Dec, DT in both primal and dual spaces.
For LV Dec and DT, we use the corresponding codes in \texttt{LattE}.

It is evident that when computing the volume of a polytope, the LV Dec may not be the optimal approach, and the use of Primal DT should be avoided, at least when $r$ is much smaller than $n$. The \texttt{Simpcone} algorithm yields a slightly larger number of simplicial cones compared to Dual DT, but demonstrates the fastest runtime in the current implementation.   The DT implemented in \texttt{LattE} exhibits high space complexity and is prone to termination when the number of extreme rays of the cone  is big.

\subsection{Some random polytopes}
We generate some random polytopes defined by $Ax=\b$ for comparison in Table \ref{table-randmat}.
For random matrix $A$ and random vector $\b$, the polytope $\P$ tends to be a simple polytope, since each $r\times r$ minor is very likely to be nonsingular.
See the column for LV.
We only list the number of cones obtained by the algorithms. From the table, we see that Dual DT seems to produce the smallest number of cones. However, our Simpcone has great flexibility. See Remark \ref{rem-Simpcone}.

\begin{table}[!h]
\scalebox{0.9}{
\begin{tabular}{clll}
\toprule
Order& Volume &  Time  \\ \midrule
4                 & 352/(9!)                                                     & 0.006s.     \\ \midrule
5                 &  4718075/(16!)                                               & 0.214s.     \\ \midrule
6                 &  14666561365176/(25!)                                        & 24.658s.                             \\ \midrule
7                 & ${17832560768358341943028}/({36!})$                          &  $ \thickapprox$ 1h20min                 \\ \bottomrule
\end{tabular}} \caption{ Compute the volume of Birkhoff polytope by Simpcone.}  \label{table-Birkhoof}
\end{table}

\begin{table}[!h]
\centering
\scalebox{1}{
\begin{tabular}{clll}
\toprule
Order & Volume & Time &  \\ \midrule
4                 &$ \frac{21}{2 \cdot(7!)}$                                                          & 0.007s     &  \\ \midrule
5                 &$\frac{31850613387721}{1428840000\cdot(14!)}$                                  & 0.833s    &  \\ \midrule
6                 &$\frac{33747011624270182119839837}{1036994918400000 \cdot(23!) }$            & 178.157s      &  \\ \bottomrule
\end{tabular}}  \caption{Compute the volume of Magic square polytope.}  \label{table-Magic}
\end{table}

\begin{table}[!h]
\centering
\scalebox{1}{
\begin{tabular}{|l|l|l|l|l|l|}
\hline
                                                                                   & Order     & 4          & 5            & 6              & 7        \\ \hline
\multirow{4}{*}{\begin{tabular}[c]{@{}l@{}}Birkhoof\\ polytope\end{tabular}}       & Simpcone     & 92(0.009s) & 3138(0.109s) & 199668(10.3s)  & 20051022(38m) \\ \cline{2-6}
                                                                                   &LV Dec     &384(0.06s)  & 15000(0.95s) & 933120(54.1s) & 84707280($\geq$4h) \\ \cline{2-6}
                                                                                   & Primal DT & 346(0.07s)  &  failed      & failed       &   failed   \\ \cline{2-6}
                                                                                   & Dual DT   & 96(0.04s)  & 3000(0.45s)  & 153656(19.1s) &   failed      \\  \hline\hline
\multirow{4}{*}{\begin{tabular}[c]{@{}l@{}}Magic\\ Square\\ polytope\end{tabular}} & Simpcone                      & 62(0.004s)  & 8869(0.46s)  & 1022153(88.2s)  &          \\ \cline{2-6}
                                                                                   & LV Dec& 200(0.06s) & 20528(1.78s)& 2106279(308.9s)&          \\ \cline{2-6}
                                                                                   & Primal DT &268(0.07s)   &failed         & failed               &          \\ \cline{2-6}
                                                                                   & Dual DT   & 58(0.06s)  & 6628(0.58s)  & 441721(140.6s) &          \\ \hline
\end{tabular}}\caption{The number of simplicial cones  and the  running time by different algorithms.} \label{table-Dec}
\end{table}

\begin{table}[!h]
\centering
\scalebox{0.9}{
\begin{tabular}{|l|l|l|l|l|l|}
\hline
Size($r,n$) & \begin{tabular}[c]{@{}l@{}}\# of extreme ray\\ (Primal // Dual)\end{tabular} & \# Simpcone & \# Primal DT & \# Dual DT & LV \\ \hline
10,14     & 10//7                                                                        & 18          & 9            & 6          &  10  \\ \hline
10,14     & 8//6                                                                         & 8           & 5            & 5          &  8  \\ \hline
10,14     & 14//9                                                                        & 12          & 15           & 13         &  14  \\ \hline
10,15     & 34//13                                                                       & 48          & 80           & 44         &  34  \\ \hline
10,15     & 11//7                                                                        & 13          & 8            & 7          &  11  \\ \hline
10,16     & 36//12                                                                       & 11          & 106          & 65         &  36  \\ \hline
10,16     & 22//9                                                                        & 185         & 50           & 22         & 22   \\ \hline
10,16     & 30//10                                                                       & 15          & 96           & 22         &  30  \\ \hline
10,17     & 15//8                                                                        & 11          & 15           & 5          &  15  \\ \hline
10,17     & 92//14                                                                       & 159         & 1192         & 126        &  92  \\ \hline
10,18     & 68//13                                                                       & 308         & 646          & 84         &  68  \\ \hline
10,18     & 14//9                                                                        & 322         & 7            & 7          &  14  \\ \hline
10,19     & 111//14                                                                      & 1986        & 3036         & 185        &  111  \\ \hline
10,19     & 165//14                                                                      & 451         & 8651         & 157        &  165  \\ \hline
10,20     & 30//11                                                                       & 118         & 126          & 5          &  30  \\ \hline
10,20     & 52//13                                                                       & 383         & 446          & 82         &  52  \\ \hline
10,21     & 417//18                                                                      & 827         & failed   & 681        &  417  \\ \hline
\end{tabular}}\caption{The polytope $\P =\{x\geq0 : Ax=b\}$ with some random matrix $A$ and random vector $b$.}\label{table-randmat}
\end{table}

\section{Concluding Remark}\label{sec:ConcludingRemark}
It is not new to use Ehrhart theory to compute the relative volume of a rational convex $d$-polytope.
In \cite[Chap~3]{beck2007computing}, it was suggested to use interpolation to compute $Ehr_\P(t)$ and then obtain $\vol(\P)$.
In \cite{barvinok2006computing}, Barvinok give a fast algorithm to compute the highest $k$ coefficients of $L_\P(s)$. In particular,
if $\P$ is a simplex, Barvinok's algorithm is polynomial when $d$ is part of the input but $k$ is fixed.
The idea was later extended for a weighted version in \cite{Baldoni2012} and has an implementation. But no closed formula was known.

This work is based on a decomposition of $\sigma_{ \mathrm{cone}(\P)}(\y;t)$ as in \eqref{equ-Ehrseries}. Such a decomposition was studied in Algebraic Combinatorics, and in Computational Geometry.
In Algebraic Combinatorics $ \sigma_{\mathrm{cone}(\P)}(\y;t)$ is just the generating function for nonnegative integer solution to a linear Diophantine system and thus can be handled by MacMahon's partition analysis. Along this line, packages $\texttt{Omega}$ \cite{TheOmegaPackage3}, \texttt{Ell} \cite{xin2004fast}, \texttt{CTEuclid} \cite{xin2015euclid} can be used to obtain the decomposition \eqref{equ-Ehrseries}. See \cite{xin2015euclid} for details and further references;
In Computational Geometry, Barvinok's algorithm, implemented by \texttt{LattE} (see references therein for related algorithms), can be used to obtain a decomposition into simplicial cones and then unimodular cones. Each $F_i(\y;t)$ so obtained corresponds to a unimodular cone.

Our basic result is Theorem \ref{theo-vol-formula}, which is obtained using constant term extraction, especially the ``dispelling the slack variable" process in \cite{xin2015euclid}.
The formula suggests that we only need simplicial cone decomposition. This leads to closed formulas in Theorems \ref{general-full-volcompute}, \ref{general-volcompute} and \ref{thm-rvol} in various cases.
We have made a Maple procedure  implementing the ideas in Section \ref{sec:simpcone}, which can be downloaded from the following link \href{https://pan.baidu.com/s/1Ymq-5sPnRmuDb5fnXxEyqQ }{Simpcone} (password: VolP).
Computer experiment confirms our formula, but the performance is bad. We turn to make a C++ procedure with good performance.

The ideas in this paper may naturally extends for weighted Ehrhart quasi-polynomials, and hence may lead to applications to certain integral over polytopes.
The ideas might apply to computing the highest coefficients of $L_\P(s)$.

\noindent
{\small \textbf{Acknowledgements:}
The authors would like to thank Matthias Beck for helpful suggestions.
This work was partially supported by the National Natural Science Foundation of China [12071311].

%\bibliographystyle{plain}
%
%\bibliography{Todd}

\begin{thebibliography}{10}

\bibitem{TheOmegaPackage3}
George~E. Andrews, Peter Paule, and Axel Riese.
\newblock {MacMahon's partition analysis \texttt{III}: the Omega package}.
\newblock {\em European Journal of Combinatorics}, 22(7):887--904, 2001.

\bibitem{Baldoni2012}
Velleda Baldoni, Nicole Berline, Jes{\'u}s~A De~Loera, Matthias K{\"o}ppe, and
  Mich{\`e}le Vergne.
\newblock {Computation of the highest coefficients of weighted Ehrhart
  quasi-polynomials of rational polyhedra}.
\newblock {\em Foundations of Computational Mathematics}, 12:435--469, 2012.

\bibitem{barvinok2006computing}
Alexander Barvinok.
\newblock {Computing the Ehrhart quasi-polynomial of a rational simplex}.
\newblock {\em Mathematics of Computation}, 75(255):1449--1466, 2006.

\bibitem{barvinokPommersheim1999summary}
Alexander Barvinok and James~E Pommersheim.
\newblock An algorithmic theory of lattice points in polyhedra.
\newblock {\em New perspectives in algebraic combinatorics}, 38:91--147, 1999.

\bibitem{beck2003ehrhart}
Matthias Beck and Dennis Pixton.
\newblock {The Ehrhart polynomial of the Birkhoff polytope}.
\newblock {\em Discrete $\&$ Computational Geometry}, 30(4):623--637, 2003.

\bibitem{beck2007computing}
Matthias Beck and Sinai Robins.
\newblock {\em Computing the Continuous Discretely}, volume~61.
\newblock Springer, 2007.

\bibitem{breuer2017polyhedral}
Felix Breuer and Zafeirakis Zafeirakopoulos.
\newblock {Polyhedral Omega: a new algorithm for solving linear diophantine
  systems}.
\newblock {\em Annals of Combinatorics}, 21:211--280, 2017.

\bibitem{brion1988points}
Michel Brion.
\newblock Points entiers dans les poly{\`e}dres convexes.
\newblock 21(4):653--663, 1988.

\bibitem{bueler2000exact}
Benno B{\"u}eler, Andreas Enge, and Komei Fukuda.
\newblock Exact volume computation for polytopes: A practical study.
\newblock {\em Polytopes-Combinations and Computation}, 29:131, 2000.

\bibitem{de2004effectiveLattE}
Jes{\'u}s~A De~Loera, Raymond Hemmecke, Jeremiah Tauzer, and Ruriko Yoshida.
\newblock Effective lattice point counting in rational convex polytopes.
\newblock {\em Journal of symbolic computation}, 38(4):1273--1302, 2004.

\bibitem{de1995triangulations}
Jesus~Antonio De~Loera.
\newblock {\em {Triangulations of Polytopes and Computational Algebra}}.
\newblock Ph. D. thesis, Cornell University, 1995.

\bibitem{ehrhart1974polyn}
Eugene Ehrhart.
\newblock {\em Polyn{\^o}mes Arithm{\'e}tiques et M{\'e}thode Des Poly{\`e}dres
  en Combinatoire}.
\newblock International series of numerical mathematics. Institut de recherche
  math{\'e}matique avanc{\'e}e, 1974.

\bibitem{glazyrin2012lower}
Alexey Glazyrin.
\newblock Lower bounds for the simplexity of the n-cube.
\newblock {\em Discrete Mathematics}, 312(24):3656--3662, 2012.

\bibitem{gritzmann1994complexity}
Peter Gritzmann and Victor Klee.
\newblock {On the complexity of some basic problems in computational convexity:
  II. Volume and mixed volumes}.
\newblock In {\em Polytopes: abstract, convex and computational}, pages
  373--466. Springer, 1994.

\bibitem{Khachiyan1993}
Leonid Khachiyan.
\newblock {\em Complexity of Polytope Volume Computation}, pages 91--101.
\newblock Springer Berlin Heidelberg, 1993.

\bibitem{lawrence1991polytope}
Jim Lawrence.
\newblock Polytope volume computation.
\newblock {\em Mathematics of computation}, 57(195):259--271, 1991.

\bibitem{lee1991reg}
Carl~W Lee.
\newblock Regular triangulations of convex polytopes.
\newblock {\em Applied Geometry and Discrete Mathematics -- The Victor Klee
  Festschrift}, 4:443--456, 1991.

\bibitem{lee1997subdivisions}
Carl~W Lee.
\newblock Subdivisions and triangulations of polytopes.
\newblock In {\em Handbook of discrete and computational geometry}, pages
  271--290. 1997.

\bibitem{sack1999handbook}
J{\"o}rg-R{\"u}diger Sack and Jorge Urrutia.
\newblock {\em Handbook of computational geometry}.
\newblock Elsevier, 1999.

\bibitem{stanley1974combinatorial}
Richard~P Stanley.
\newblock Combinatorial reciprocity theorems.
\newblock {\em Advances in Mathematics}, 14(2):194--253, 1974.

\bibitem{stanley2011EC1}
Richard~P Stanley.
\newblock {\em Enumerative Combinatorics Volume I second edition}.
\newblock Cambridge University Press, 2011.

\bibitem{verdoolaege2005computation}
Sven Verdoolaege, Kevin~M Woods, Maurice Bruynooghe, and Ronald Cools.
\newblock Computation and manipulation of enumerators of integer projections of
  parametric polytopes.
\newblock {\em CW Reports}, pages 104--104, 2005.

\bibitem{xin2004fast}
Guoce Xin.
\newblock {A fast algorithm for MacMahon's partition analysis}.
\newblock {\em The Electronic Journal of Combinatorics}, 11(R58):1, 2004.

\bibitem{xin2007generalization}
Guoce Xin.
\newblock {Generalization of Stanley's monster reciprocity theorem}.
\newblock {\em Journal of Combinatorial Theory, Series A}, 114(8):1526--1544,
  2007.

\bibitem{xin2015euclid}
Guoce Xin.
\newblock {A Euclid style algorithm for MacMahon's partition analysis}.
\newblock {\em Journal of Combinatorial Theory, Series A}, 131:32--60, 2015.

\bibitem{xinxudedekindsums}
Guoce Xin and Xinyu Xu.
\newblock {A polynomial time algorithm for calculating Fourier-Dedekind sums}.
\newblock {\em arXiv preprint arXiv:2303.01185}, 2023.

\bibitem{conedec}
Guoce Xin, Xinyu Xu, and Zihao Zhang.
\newblock A combinatorial simplicial cone decomposition.
\newblock {\em In preparation}.

\bibitem{XinTodd}
Guoce Xin, Yingrui Zhang, and Zihao Zhang.
\newblock Fast evaluation of generalized todd polynomials: Applications to
  macmahon's partition analysis and integer programming.
\newblock {\em arXiv preprint arXiv:2304.13323}, 2023.

\end{thebibliography}

\end{document}